\documentclass{amsart}
\usepackage{amsfonts}
\usepackage{amsmath}
\usepackage{amssymb}
\usepackage{graphicx}
\usepackage{version}
\usepackage{color}

\newcommand{\R}{{\mathbb R}}
\newcommand{\x}{{\mathbb X}}
\newcommand{\N}{{\mathbb N}}
\newcommand{\So}{\mathbb{S}^1}

\newcommand{\e}{\varepsilon}

\newcommand{\id}{\operatorname{id}}

\newcommand{\K}{\mathcal{K}}
\newcommand{\F}{F}
\newcommand{\G}{G}
\newcommand{\mH}{H}
\newcommand{\on}{\operatorname}

\theoremstyle{plain}
\newtheorem{theorem}{Theorem}[section]
\newtheorem{cor}{Corollary}[section]
\newtheorem{lemma}{Lemma}[section]
\newtheorem{proposition}{Proposition}[section]
\theoremstyle{definition}
\newtheorem{definition}{Definition}[section]
\newtheorem{example}{Example}[section]
\newtheorem{remark}{Remark}[section]
\newtheorem{conjecture}{Conjecture}[section]
\newtheorem{question}{Question}[section]
\numberwithin{equation}{section}
\excludeversion{comment1}
\begin{document}

\title[Transition Phenomena for the Attractor of an Iterated Function System] {Transition Phenomena for the Attractor of an Iterated Function System}
\author[K. Le\'sniak]{Krzysztof Le\'sniak}
\address{Faculty of Mathematics and Computer Science, 
	Nicolaus Copernicus University in Toru\'{n},
	Chopina 12/18, 87-100 Toru\'{n}, 
	Poland}
\email{much@mat.umk.pl}
\author[N. Snigireva]{Nina Snigireva}
\address{School of Mathematical Sciences,
	Dublin City University,
	Glasnevin, Dublin 9, 
	Ireland}
\email{nina.snigireva@dcu.ie}
\author[F. Strobin]{Filip Strobin}
\address{Institute of Mathematics, 
	Lodz University of Technology, 
	W\'{o}lcza\'{n}ska 215, 90-924 \L\'{o}d\'{z}, 
	Poland}
\email{filip.strobin@p.lodz.pl}
\author[A. Vince]{Andrew Vince}
\address{Department of Mathematics \\ University of Florida \\ USA}
\email{\tt  avince@ufl.edu} 
\subjclass[2010]{28A80}
\keywords{iterated function system, attractor}
\thanks{This work was partially supported by a grant from the Simons Foundation (322515 to Andrew Vince).}
\begin{abstract}  
Iterated function systems (IFSs) and their attractors 
have been central to the theory of fractal geometry almost 
from its inception.  And contractivity of the functions 
in the IFS has been central to 
the theory of iterated functions systems.  
If the functions in the IFS are contractions, then the IFS 
is guaranteed to have a unique attractor. 
Recently, however, there has been an interest in what occurs to
the attractor at the boundary between contractvity and 
expansion of the IFS. That is the subject of this paper.  
For a family $F_t$ of IFSs depending on a real parameter $t>0$, the
existence and properties of two types of transition attractors,
called the lower transition attractor $A_{\bullet}$ and the 
upper transition attractor $A^{\bullet}$, are investigated. 
A main theorem states that, for a wide class of IFS families, 
 there is a threshold $t_0$
such that the IFS $F_t$ has a unique attractor $A_t$ for $t<t_0$ 
and no attractor for $t>t_0$. At the threshold $t_0$, 
there is an $F_{t_0}$-invariant set $A^{\bullet}$ such that
$A^{\bullet} = \lim_{t\rightarrow t_0} A_t$.
\end{abstract}

\maketitle

\section{Introduction} \label{sec:intro}

Iterated function systems (IFSs) and their attractors 
have been central to the theory of fractal geometry almost 
from its inception.  And contractivity of the functions 
in the IFS has been central to 
the theory of iterated functions systems.  
If the functions in the IFS are contractions, then the IFS 
is guaranteed to have a unique attractor 
(see Hutchinson's seminal Theorem~\ref{thm:H} below).  
Recently, however, there has been an interest in what occurs to
 the attractor at the boundary between contractvity and 
expansion of the IFS. That is the subject of this paper. 

Let $\x$ denote a compete metric space with metric $d(\cdot, \cdot)$.
A finite iterated function system (IFS) is a set  
\[
F:=\{f_{1},f_{2}, \dots ,f_{N}\}
\]
of $N\geq 2$ continuous functions from $\mathbb X$ to itself. 
An IFS is {\it affine} if its functions are invertible affine functions 
on $d$-dimentional Euclidean space $\R^d$, 
{\it projective} if its functions are non-singular projective functions 
on $d$-dimentional real projective space $\R\mathbb P ^d$, 
and {\it M\"obius} if its functions are M\"obius transformations 
on the extended complex plane $\mathbb{C}\cup \{\infty\}$, 
i.e., on the Riemann sphere.  An affine IFS all of whose functions are 
similarities is referred to as a {\it similarity IFS}.  An affine IFS 
all of whose functions are non-singular linear maps
is refer to as a {\it linear IFS}.

For a function $f : \x \rightarrow \x$, let 
\[
Lip(f,d) := \sup_{x\neq y} \frac {d(f(x),f(y))}{d(x,y)} 
\]
denote the {\it Lipschitz constant} of $f$ with respect to the metric $d$.  
Let
\[
Lip(F,d) := \max_{f\in F} \, Lip(f,d).
\]
A function $f$ is {\it Lipschitz} if $Lip(f,d) < \infty$, 
and an IFS $F$ is {\it Lipschitz} if $Lip(F,d) < \infty$. 
A function $f$ is a {\it contraction} with respect to $d$ 
if $Lip(f,d) < 1$, and is {\it nonexpansive} if  $Lip(f,d) \leq 1$.  

\begin{definition} 
An IFS $F$ on $\x$ is {\bf contractive},  if there is 
an {\bf equivalent metric} $d'$ on $\x$, i.e., 
a metric $d'$ giving the same topology as the original metric $d$, 
such that $\mathbb X$ remains complete with respect to 
$d'$ and $Lip(F,d') <1$. \end{definition}

\noindent Allowing metrics topologically equivalent 
to the original metric is essential, for example, 
to the validity of Theorem~\ref{thm:apm} below. Also see Example~\ref{ex:evcontr}.
\vskip 2mm

For the collection $\K(\x)$ of non-empty compact subsets 
of $\x$, the classical Hutchinson operator 
$F\,:\,{\K}(\x)\rightarrow {\K}(\x)$ is given by 
\[ 
F(K):=\bigcup_{f \in F} f (K).
\]
By abuse of language, the same notation $F$ is used for the IFS, 
the set of functions in the IFS, and for the Hutchinson operator; 
the meaning should be clear from the context. 
A compact set $A\subseteq \x$ is the (strict) {\bf attractor} of $F$ 
if there is an open neighborhood $U\supseteq A$ such that 
\begin{itemize} 
\item (\textit{invariance}) $F(A) = A$, and
\item (\textit{attraction}) $A=\lim_{n\to\infty}F^{(n)}(K)$,
\end{itemize}
where $F^{(n)}$ denotes the $n$-fold composition, 
the limit is with respect to the Hausdorff metric 
and is independent of the non-empty compact set $K\subseteq U$. 
So the attractor is the Banach fixed point of 
the Hutchinson operator on $\K(U)$.    
The largest such set $U$ is called the {\it basin} of $F$.  

\begin{theorem}[Hutchinson \cite{hutchinson}] \label{thm:H} 
A contractive IFS on a complete metric space $\x$
has a unique attractor with basin $\x$.  
\end{theorem} 

In classical IFS theory, it is assumed that the functions 
in the IFS are contractions, a natural assumption 
in light of Hutchinson's theorem.  More recently, however, 
papers have appeared on IFS attractors assuming average contractivity 
(see \cite{st} for a survey), on IFSs that are weakly contractive 
(see, for example, \cite{lss}), and on relaxing the definition 
of an attractor; see, for example, \cite{lm1,lm2} 
in which the notion of a semiattractor is introduced 
to explain the nature of supports of invariant measures 
of average contractive IFSs  \cite{BarnsleyElton}. 
This paper is concerned with attractor phenomena 
at the transition between contractivity 
and expansion of a one-parameter IFS family, between the existence 
and non-existence of an attractor. To illustrate this kind of 
transition phenomena, consider the following family $F_t$ of IFSs 
that depends on a real parameter $t > 0$, which is based on \cite[Example 1.1]{V}. 

\begin{example} \label{ex:intro3d} 

In $\R^3$ let $F_t := \{f_{(i,t)}, 1\leq i\leq 2\}$ be the one-parameter 
affine family where $f_{(i,t)}(v) = t\, L_i(v-q_i) +q_i$, and where
\[
L_1 = 
\begin{pmatrix} 
\frac{\sqrt{2}}{2} & -\;\frac{\sqrt{2}}{2} & 0 \\ 
\frac{\sqrt{2}}{2} & \frac{\sqrt{2}}{2} & 0 \\
0 & 0 & 1 \\
\end{pmatrix}, 
\]
is the rotation by $\pi/4$ about the $z$-axis and 
$q_1=(0,0,2)$ is a fixed point of $L_1$
outside the $xy$-plane; 
$L_2=0.4 \, L_1$ and $q_2=(1,0,0)$.
\vskip 2mm

\noindent For $t\in (0,1)$, the IFS $F_t$ is contractive
and has an attractor $A_t$. Figure~\ref{fig:morph-3d} shows 
views of $A_t$ for $t =0.9$ and $t=0.96$. For $t\geq 1$, 
the IFS $F_t$ fails to be contractive and has no attractor.  
The value $t=1$ is called a {\it threshold}, defined precisely 
in Definition \ref{def:t} below.    

The question arises as to the nature of the transition 
at the threshold $t=1$. In this example, intriguing $F_1$-invariant 
sets occur. We refer to such sets as  {\bf transition attractors}, 
and we consider two types: 
{\it lower transition attractors}, denoted $A_{\bullet}$,
and {\it upper transition attractors}, denoted $A^{\bullet}$. 
Precise definitions appear in Section~\ref{sec:uls}. 
The terminology ``upper" and ``lower" is due to the fact that,
for appropriately defined one-parameter families, it is the case that
$A_{\bullet} \subseteq A^{\bullet}$.  

Figure~\ref{fig:tr2-3d} shows the lower transition
attractor and Figure~\ref{fig:tr3-3d} shows the upper transition attractor for the IFS family of Example~\ref{ex:intro3d}.  
The subject of transition attractors, in two guises, 
was introduced independently in \cite{KS} and \cite{V}. 
\end{example}

\begin{definition} \label{def:t}
A {\bf one-parameter family} is an IFS family
\begin{equation*}  
F_t := \{f_{(1,t)}, f_{(2,t)} \dots, f_{(N,t)}\}
\end{equation*} 
parametrized by a real number $t\in (0, \infty)$.  The intuition 
is that, the nearer the parameter $t$ is to $0$, the more 
contractive the functions in the IFS, and as $t$ increases, 
the functions in $F_t$ become less contractive.  
A real number $t_0$ is called the {\bf threshold} for 
the existence of an attractor of $F_t$ if $F_t$ has an attractor 
for $t<t_0$ but fails to have an attractor for $t>t_0$.  
\end{definition}

\begin{figure}[htb]  
\vskip 3mm
\includegraphics[width=5cm, keepaspectratio]{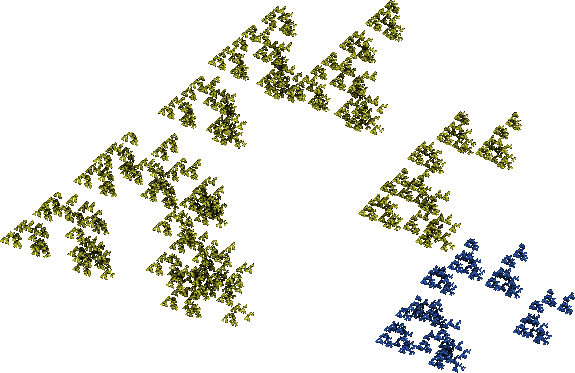} \hskip 10mm
\includegraphics[width=5cm, keepaspectratio]{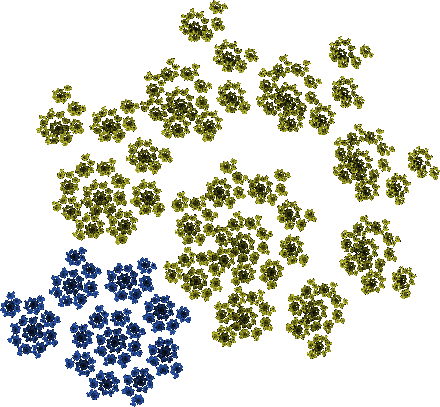}
\vskip 5mm 
\includegraphics[width=5cm, keepaspectratio]{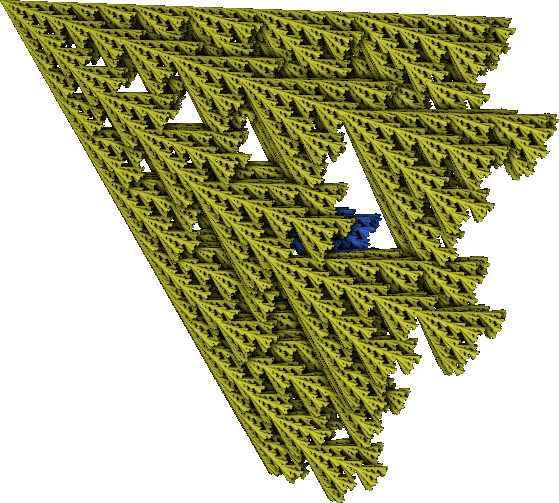}  \hskip 10mm
\includegraphics[width=5cm, keepaspectratio]{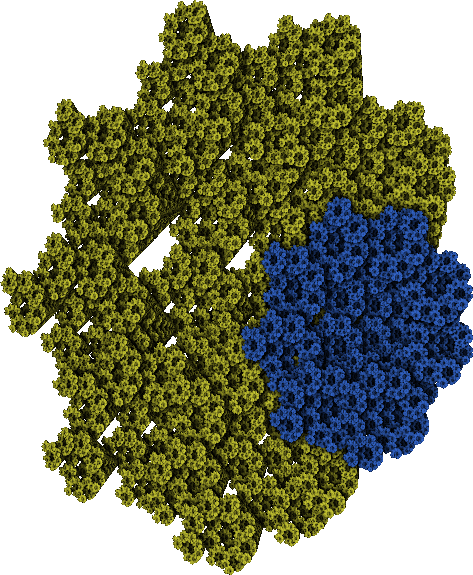}
\caption{The attractor $A_t$ for the one-parameter 
affine family $F_t$ of Example~\ref{ex:intro3d} 
for parameter values $t = .9$ (top line), $t= .96$ (bottom line); 
side and bottom view of a fractal "cone". The green and blue colours indicate the image of the attractor
under the two maps of the IFS.  Note that $f_{(1,t)}(A)\cap f_{(2,t)}(A)\neq \emptyset$.
}
\label{fig:morph-3d}
\end{figure}

\begin{figure}[th]  
\vskip 11mm  \includegraphics[width=5cm, height=7cm, keepaspectratio]{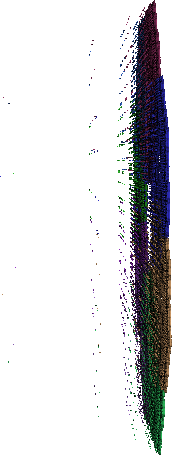}
\hskip 6mm  \includegraphics[width=5cm, keepaspectratio]{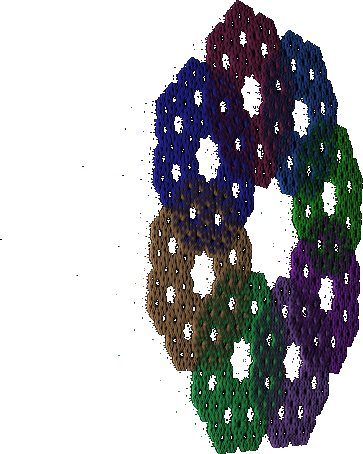} 
\caption{The lower transition attractor of Example~\ref{ex:intro3d} - side and bottom view of a fractal ``cone".}\label{fig:tr2-3d}
\end{figure}

\begin{figure}[th]  
	\vskip 11mm  \includegraphics[width=5cm, height=7cm, keepaspectratio]{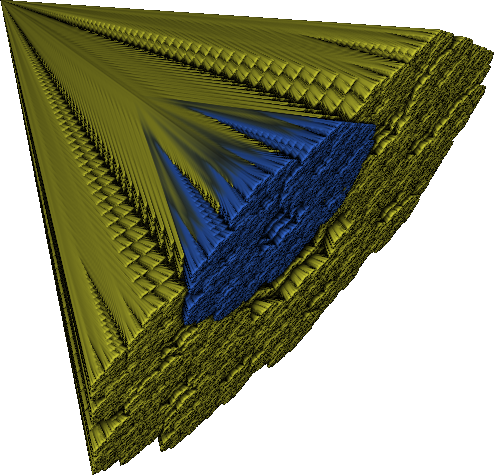}
	
\caption{ The upper transition attractor of Example~\ref{ex:intro3d}.}
 \label{fig:tr3-3d}
\end{figure}

The main open question in \cite{V} was the following.

\begin{question} \label{conj:A} 
If $A_t, \, t\in (0,t_0),$ denotes the attractor 
of a one-paramter  family $F_t$ of affine IFSs with threshold $t_0$, 
what conditions on $F_t$ guarantee the existence of 
a unique upper transition attractor, i.e.,
a compact $F_{t_0}$-invariant set $A^{\bullet}$ such that   
\[A^{\bullet}= \lim_{t\rightarrow t_0} A_t.\]
\end{question}

In \cite{V} certain conditions on a one-parameter family 
of affine functions were conjectured to guarantee such 
an upper transition attractor $A^{\bullet}$.  A main result 
in this paper is a proof of a strong version of that conjecture 
in the setting of a real Banach space.

\section{Organization - Previous and New Results} 

The paper is organized as follows.  
\vskip 1mm

\begin{itemize}
\item (Section~\ref{sec:ca}: Contractivity, Attractors, and Thresholds)

 In this paper we are interested in transitions for one-parameter IFS families 
$F_t$ at thresholds between the existence and non-existence of an attractor, 
between contractivity and of non-contractivity of $F_t$.  For affine families 
(Definition~\ref{def:af}) like that of  Example~\ref{ex:intro3d}, it is known that
there does exist a single threshold (Theorem~\ref{thm:af}).  This is also the case
for our main Theorem~\ref{thm:fildec15}.  

From the origin of IFS theory, the existence of an attractor 
has been associated with the contractivity of the IFS. 
The precise relationship, however, has not been completely delineated.  
The issue involves the converse of Hutchinson's Theorem~\ref{thm:H}.  
For an IFS $F$ on a complete metric space, Hutchinson's theorem states 
that contractivity of an IFS is a sufficient condition for the existence 
of a unique attractor.  When the IFS contains only one mapping, 
the converse (which is a converse to the Banach Contraction 
Mapping Theorem) was proved by Jan\"os \cite{LJ} and 
by Leader \cite{SL}.  A converse is known to hold for affine, 
projective and M\"obius IFSs (Theorem~\ref{thm:apm}). In general, however, there are examples of IFSs which admit attractors yet there is no
equivalent metric with respect to which the functions in the IFS are contractions.

In Examples~\ref{ex:S1}, \ref{ex:S2}, and  \ref{ex:S3}, the IFS $F$ admits a unique attractor 
but $Lip(F,d) > 1$ for all equivalent metrics $d$ on $\x$. 
\vskip 2mm

\item (Section~\ref{sec:uls}:  Lower Transition Attractors, Upper Transition Attractors, and Semiattractors)
\vskip 2mm

If a threshold $t_0$ for the existence of an attractor 
does exist for a one-parameter family, then 
the question arises as to what occurs at this threshold. 
For some one-parameter families $F_t$, there exist 
intriguing $F_{t_0}$-invariant sets referred to as 
the lower and upper transition attractors 
(Definitions~\ref{def:l} and \ref{def:u}). The existence and some properties of a lower attractor is the subject of  
Theorem~\ref{thm:lowA},  Corollary~\ref{cor:compact}, and 
Proposition~\ref{prop:sym}. Example~\ref{ex:dream3d} illustrates these results. 
The existence and some properties of an upper transition attractor is the subject of 
Theorem~\ref{thm:ut2} and Proposition~\ref{prop:sym}. The relationship between
the lower and upper transition attractors is the subject of Theorem~\ref{thm:LU}. 

The lower transition attractor of certain one-parameter families 
is shown in statement (iv) of Theorem~\ref{thm:lowA}   to be the semiattractor 
(Definition~\ref{def:s}) of an associated single IFS. 
Properties of semiattractors are contained in Theorem~\ref{thm:semi}. 
\vskip 2mm

\item (Section~\ref{sec:ut}: The Existence of a Unique Upper Transition Attractor)
\vskip 1mm

Theorem~\ref{thm:fildec15}, the main result of the paper, 
provides an answer to Question~\ref{conj:A} in the introduction - giving conditions that guarantee a {\it unique} upper transition attractor at a threshold for
the existence of an attractor.
The existence of a unique upper transition attractor was 
conjectured for a special type of one-parameter similarity family 
in Euclidean space in \cite{V}.  
The underlying space in Theorem~\ref{thm:fildec15} is the more general Banach space,
and the one-parameter families are more general than in \cite{V}.     Examples~\ref{ex:two}, \ref{ex:three}, \ref{ex:four}, and \ref{ex:inf}
 show that the assumptions in the hypothesis of Theorem~\ref{thm:fildec15} are all necessary, at least in the infinite dimensional case.  Question~\ref{ques:per} in Section~\ref{sec:op} asks whether the ``periodicity" assumption in Theorem~\ref{thm:fildec15} can be dropped assuming a less exotic space. 
\vskip 2mm

\item (Section~\ref{sec:op}:  Open Problems)
\vskip 1mm

There remain questions and conjectures about 
thresholds and transition attractors that remain open. 
Several are posed in this section. 

\end{itemize}

\section{Contractivity, attractors, and thresholds} \label{sec:ca}
 
For an IFS on a complete metric space, the converse of 
Hutchinson's Theorem~\ref{thm:H}, does not, in general, hold.  
Examples~\ref{ex:S1}, \ref{ex:S2}, and \ref{ex:S3} 
are provided below. These examples 
not withstanding, a converse does hold in the affine, M\"obius, 
and projective cases. 

\begin{theorem}[\cite{ABVW,bv1,V2}] \label{thm:apm}  
An affine, M\"obius, or projective IFS can have 
at most one attractor.  Moreover,
\begin{enumerate}
\item An affine IFS $F$ has an attractor if and only if 
$F$ is contractive on $\R^d$.
\item A M\"obius IFS $F$ has an attractor 
$A \neq \mathbb C\cup \{\infty\}$  if and only if 
$F$ is contractive on an open set 
whose closure is not $\mathbb C\cup \{\infty\}$.
\item A projective IFS $F$ has an attractor 
that avoids some hyperplane if and only if 
$F$ is contractive on the closure of some open set.
\end{enumerate}
\end{theorem}

For an IFS $F$ the distinction between all functions in $F$ being contractions and $F$ contractive must be emphasized.  See Example~\ref{ex:evcontr} below. 

\begin{example}[A family of contractive affine IFSs $F_t$ on $\R^2$ such that the functions in $F_t$ are not contractions with respect to the Euclidean metric cf. \cite{lss} Example 6.3.]
\label{ex:evcontr} 

 Define $F_t := \{f_{(1,t)}, f_{(2,t)} \}$, where
\[
f_{(1,t)}(v) = \begin{pmatrix}   0& \kappa_1  t \\ \lambda_1/ t & 0 \end{pmatrix} \,v \,, 
\qquad \qquad
f_{(2,t)}(v) = \begin{pmatrix} 0 &  \kappa_2 t \\ \lambda_2 /t & 0   \end{pmatrix}\, v + 
\begin{pmatrix}  1/\lambda_2 -t \\  1/\kappa_2 - 1/t \end{pmatrix},
\]  
where $\lambda_i \, \kappa_j <1$ for $i,j \in \{1,2\}$.

The functions in $F_t$ are contractions with respect to the Euclidean metric
only for \linebreak $\min \{ 1/\kappa_1, 1/\kappa_2\} > t > 
\max \{ \lambda_1, \lambda_2\}$.  
We claim, however, that  $F_t$ admits an attractor $A_t$ for all $t>0$.  Since
$F_t$ consists of affine functions for each $t>0$, it would then follow that
 the IFS $F_t$ is contractive 
for each $t>0$ by Theorem~\ref{thm:apm} part (1).

To see that $F_t$ has an attractor, considier the second iterate 
$F_{t}^2 := \{f_{(i,t)} \circ f_{(j,t)}: 1\leq i,j \leq 2\}$ 
of $F_t$ given by
\[
f_{(i,t)} \circ f_{(j,t)}\, (v) =  
\begin{pmatrix}  \kappa_i \, \lambda_j  &  0\\ 
0 &  \kappa_j \, \lambda_i \end{pmatrix} v + a(i,t),
\]
where the vectors $a(i,t), i = 1,2,$ are readily calculated.  
The two functions in $F_{t}^2$ are contractions 
for all $t\in (0,\infty)$ when $0<\kappa_i \lambda_j <1$ 
for $i,j \in \{1,2\}$, and therefore have an attractor 
for all $t\in (0,\infty)$. If an attractor exists 
for one of them, then $F_t$ and $F_t^2$ 
have the same attractor.  
Therefore $F_t$ admits an attractor $A_t$ for all $t$.  
The attractor of  $F_{t}^2$ is shown in  Figure \ref{fig:evcontr} 
for three values of $t$ in the case that 
$\lambda_1 = 1/4, \kappa_1= 3, \lambda_2= 1/5$ and $\kappa_2 = 2$.
The functions in $F_t$ are contractions with respect to 
the Euclidean metric only for $t\in(1/4,1/3)$, yet the functions 
in the second iterate 
$F_{t}^2$  
are contractions for all $t\in (0,\infty)$.

\begin{figure}[htb]  
\vskip 3mm
\includegraphics[width=5cm, keepaspectratio]{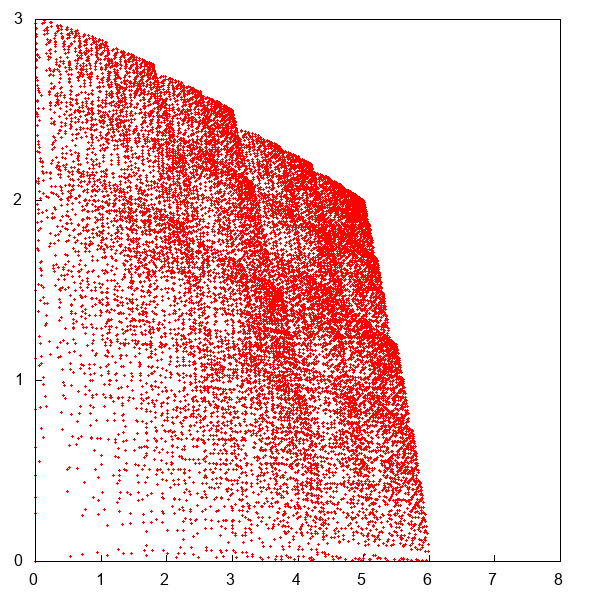}
\includegraphics[width=5cm, keepaspectratio]{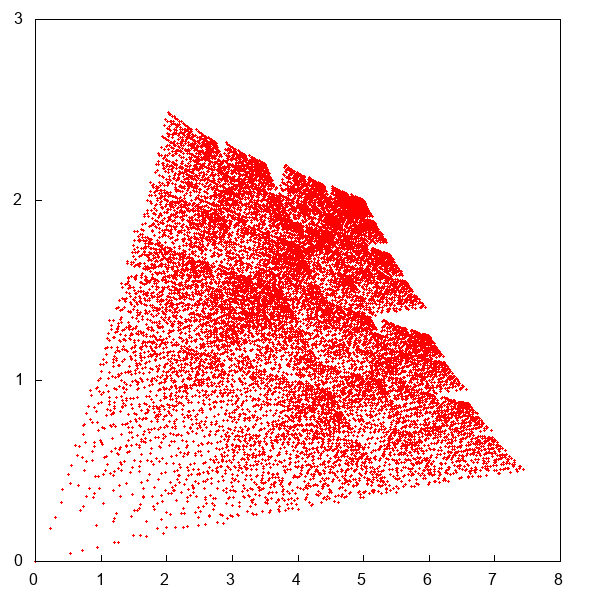}
\includegraphics[width=5cm, keepaspectratio]{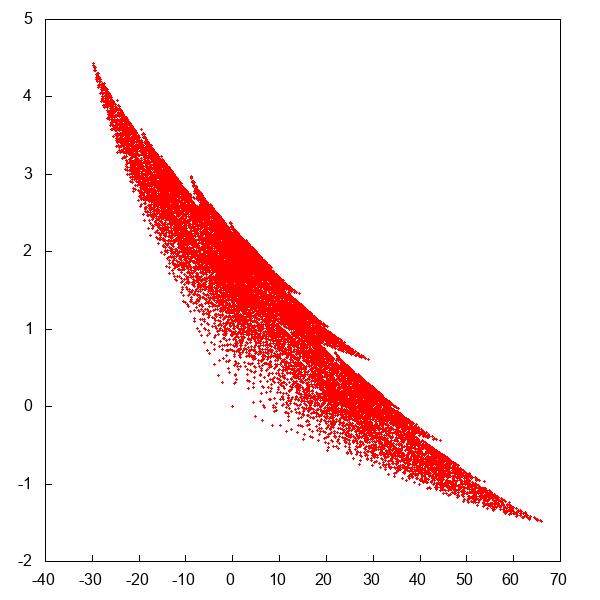} 
\caption{The attractor $A_t$ for the one-parameter 
affine family $F_t$ of Example~\ref{ex:evcontr} 
for successive parameter values $t = .5, 1, 5$.}
\label{fig:evcontr}
\end{figure}
\end{example}

There exist IFSs that have 
an attractor but are not contractive.  For the examples $F$ in \cite{bv3,lss} 
 $Lip(F,d)=1$. Our counterexamples  below are of 
\begin{enumerate}
\item an IFS $F$  on the circle $\So$ that admits a unique attractor 
but $Lip(F,d) > 1$ for all  equivalent metrics $d$ on $\x$ 
(Example~\ref{ex:S1}), 
\item a stronger counterexample of 
 an IFS $F$  on $\So$ that admits a unique 
attractor but $Lip(f,d) > 1$ 
for all $f\in F$ and all equivalent metrics $d$ on $\x$ 
(Example~\ref{ex:S2}), and 
\item an IFS on $\R^n$ that has an attractor but is not contractive (Example~\ref{ex:S3}).  
\end{enumerate} 

Let $\So$ be the unit circle centered at the origin in the complex plane, 
and let $f :\So\rightarrow \So$ be the angle doubling map $f(z) = z^2$
(cf. \cite{FitzKunze}).   
Let $\rho:\So\rightarrow \So$ be the rotation map 
$\rho(z) = e^{i \alpha}\, z$ where $\alpha/2\pi$ is irrational 
and let $g(z) =  \rho\circ f(z)$.  The following proposition is
 helpful in showing the validity of the two examples.  We include a direct proof of this proposition for completeness though it can be obtained from the standard theory of topological dynamics, see Remark \ref{rem:anotherproof}.

\begin{proposition} \label{prop:Lip}  
If $d$ is any metric on $\So$ inducing 
the standard topology on $\So$, 
then $Lip(f,d) > 1$ and $Lip(g,d) > 1$.
\end{proposition}

\begin{proof}
Suppose that $Lip(f,d) \leq 1$ for some $d$. 
Abbreviate the point $e^{i \theta}$ by $z_{\theta}$. 
Then for any $z = z_{\phi} \neq 1$ we have
\[ 
d(1, z) = d(f(1), f(z_{\phi/2})) \leq 
d(1,z_{\phi/2}) = d(f(1), f(z_{\phi/4})) \leq 
d(1,z_{\phi/4}) = \cdots.
\]
Therefore  $d(1,z) \leq d(1,z_{\phi/2^n})$ for all $n\in\N$. 
Because $z_{\phi/2^n}  \to 1$ as $n\to \infty$ 
with respect to the standard topology on $\So$, 
we have that $d(1,z) < \epsilon$ for every $\epsilon >0$. 
Therefore $d(1,z)=0$ and $z =1$, a contradiction.

To see that $Lip(g,d)>1$, note that the mappings $f$ and $g$ 
are conjugate, specifically $g = \rho^{-1}\circ f\circ \rho$.  
If $d$ is any metric inducing the standard topology 
on $\So$, then the metric
\[
d'(x,y) := d(\rho^{-1}(x), \rho^{-1}(y))
\]
also induces the standard topology on $\So$. 
By the paragraph above, there exist $x',y'\in S^1$
such that 
\[
d'(f(x'),f(y')) = c \,d'(x',y'),
\]
where $c>1$. If $x=\rho^{-1}(x')$ and $y=\rho^{-1}(y')$, then
\[\begin{aligned}  
  d(g(x),g(y) ) &
  = d( \rho^{-1} \circ f\circ \rho(x), \rho^{-1} \circ f\circ \rho (y))    
  = d'(f\circ \rho (x),f\circ \rho (y)) = d'(f(x'),f(y') ) \\ 
  & = c\, d' (x',y') = c \, d(\rho^{-1}(x'), \rho^{-1}(y') ) 
  = c\, d(x,y).
\end{aligned}\]
\end{proof}

\begin{remark}\label{rem:anotherproof} One can easily see that $f$ and $g$ are locally  distance doubling with respect to the arc metric on $\So$.  Therefore
 they are topologically expanding 
(\cite{AokiHiraide} chapter 2.2 and \cite{Reddy}). 
Since the notion of a topologically expanding map on 
a compact space does not depend on the choice of metric,
this proves Proposition \ref{prop:Lip}.
Moreover, neither $f$ nor $g$ are 
locally nonexpansive at any point under any 
admissible metrization $d$ of $\So$.
\end{remark}

\begin{example}[An IFS on the circle $\So$ 
having an attractor, but with $Lip(F,d)>1$ 
for all admissible metrizations $d$ of $\So$] \label{ex:S1}

With $f$ and $\rho$ as defined above, let $F := \{f,\rho,\id\}$, 
where $\id$ is the identity map on $\So$. 
That $Lip(F,d) > 1$ follows from Proposition~\ref{prop:Lip}.  
That $\So$ is the attractor of $F$ is seen as follows. 
The invariance $F(\So) = \So$ is clear since $h$ is a rotation. 
That $\lim_{n\rightarrow \infty} F^{(n)}(z) = \So$ for any $z\in \So$ 
can be seen as follows. We have 
$\{\rho^{(m)}(z) : 0\leq m\leq n\} \subseteq F^{(n)}(z)$ and 
$\{\rho^{(m)}(z)\}_{m=0}^{\infty}$ is dense in $\So$, since $h$
is an irrational rotation.
\end{example}

\begin{example}[An IFS on the circle $\So$ 
having an attractor, but with $Lip(f,d)>1$ for all $f\in F$ 
under any admissible metrization $d$ of $\So$] \label{ex:S2}

With $f$ and $g$  has defined above, let $F := \{f,g\}$.  
Again, that $Lip(f,d) > 1$ and $Lip(g,d) > 1$ follows from 
Proposition~\ref{prop:Lip}. That $\So$ is 
the attractor of $F$ is seen as follows. 
The invariance $F(\So) = \So$ is clear since $f$ maps 
$\So$ onto itself. That $\lim_{n\rightarrow \infty} F^{(n)}(z) = \So$ 
for any $z\in \So$ can be seen as follows. 
Abbreviate the point $e^{i \theta}$ by $z_{\theta}$. 
For $(a_1, a_2,\dots, a_n) \in (\mathbb Z_2)^n$, 
denote by $f_{(a_1,a_2,\dots, a_n)} : \So\rightarrow\So$
the map given by
\[ 
f_{(a_1,a_2,\dots, a_n)}(z_{\theta}) = 
z_{(2^n \theta + \sum_{k=1}^{n} 2^{k-1} a_k \alpha )}.
\]
(That is $f_{(a_1,a_2,\dots,a_n)} = 
f_{a_n}\circ\dots\circ f_{a_2}\circ f_{a_1}$,
under identification $f_0:=f$, $f_1:=g$.)
Then for any $z = z_{\theta}$ we have
\begin{eqnarray*}
F^{(n)}(z) = \{f_{(a_1,\dots, a_n)}(z_{\theta}) : 
(a_1,\dots, a_n) \in (\mathbb{Z}_2)^n\} =
\{\rho^{(m)}(z_{2^n\theta}) : 0\leq m < 2^n \}.
\end{eqnarray*}
Therefore, for any $\epsilon>0$ there is an $n$ 
such that there is no arc on $\So$ of length $\epsilon$ 
not containing a point of $F^{(n)}(z)$. 
Therefore $\lim_{n\rightarrow \infty} F^{(n)}(z) = \So$.
\end{example}

\begin{example}[An IFS on $\R^n$ that has an attractor but is not contractive]
 \label{ex:S3}

Let $A$ be a unit cube in $\R^n$, or any other convex compact set in $\R^n$, other than a single point, that is the attractor of an IFS $F$.  Then  $A$ is a retract of $\R^n$, i.e., there exists a continuous map $r: \R^n \to A$ such that $r(\R^n)=A$ and $r$ restricted to $A$ is the identity map. 
(In fact, any set homeomorphic to a convex compact subset of a Banach space $\x$ is a retract of $\x$,  cf. \cite[Chp. I, Corollary 1.4, Definition 1.7 and Theorem 1.9.1 ]{Gorn})  Since
$A$ contains more than one point, the map $r$ cannot be a contraction with respect
to any metric equivalent to the Euclidean metric.  Now let $G = F \cup \{r\}$.  
Then $G$ is an IFS with attractor $A$.  Indeed, $F^k(S) \subseteq G^k(S) \subseteq F^k(S) \cup A$ for any non-empty $S\subset  \R^n$.  Since $r$ cannot be a contraction with respect to any metric equivalent to the Euclidean metric, the IFS $G$ is not
contractive.  
\end{example} 

 \begin{remark} The possibility of remetrization of a given IFS $\F$ 
by a metric making each map weakly contractive 
is equivalent to the existence of a coding map \cite{BKNNS, MiculescuMihail}. \end{remark}

\begin{definition} \label{def:af}  
A one-parameter family 
\[F_t := \{f_{(1,t)}, f_{(2,t)} \dots, f_{(N,t)}\}\]
whose functions have the form
\begin{equation*}   
f_{(i,t)}(x) = t\, f_i(x) + q_i, \qquad x\in\R^d
\end{equation*} 
where 
\[
F := \{f_1, f_2, \dots, f_N\} 
\qquad \qquad   \text{and} \qquad \qquad
Q := \{q_1, q_2, \dots, q_N\} 
\]
\vskip 2mm
\noindent are a set of invertible affine transformations on $\R^d$ 
and a set of vectors in $\R^d$, respectively, is called a 
{\bf one-parameter affine family}.   
\end{definition}
 
Theorem~\ref{thm:af} below states that a one-parameter affine family 
has a threshold for the existence of an attractor.  The threshold in 
Example~\ref{ex:intro3d} is $t_0=1$. See \cite{BW,DL,RS} for background 
on the joint spectral radius.   

\begin{theorem}[\cite{V}] \label{thm:af}  
For a one-parameter affine family $F_t$, let $t_0 = 1/\rho(F)$, 
where $\rho(F)$ is the joint spectral radius of 
the linear parts of the functions in $F$.  
Then $F_t$ has an attractor for $t <t_0$ and 
fails to have an attractor for $t > t_0$.  
\end{theorem}

More can be said for a {\bf linear family} $F_t$, all of whose 
maps are of the form $f_t(x) = t\, L(x)$, where $L$ is 
a non-singlular linear map. In this case, it immediately follows 
from Theorem~\ref{thm:af} that the attractor $A_t$ of $F_t$ 
is the origin, a single point, for all $t<t_0$, and there is 
no attractor for all $t>t_0$. However, the following holds. 

\begin{theorem}[\cite{bv2}] 
Let $F_t$ be an irreducible 
($F$ admits no non-trivial invariant subspace), 
one-parameter linear IFS family on $\R^d$ with threshold $t_0$. 
Then there exists a compact $F_{t_0}$-invariant set 
that is centrally symmetric, star-shaped, 
and whose affine span is $\R^d$.  
\end{theorem}

In other words, the attractor evolves with the parameter $t$ 
from trival to non-existent, blowing up only at the single 
threshold value $t=t_0$. An example in $\R^2$ is shown 
in Figure~\ref{fig:f3} for $F := \{L_1, L_2\}$ where 
\[
L_1 = \begin{pmatrix}  0.02 & 0 \\0 & 1 \end{pmatrix}, \qquad \qquad
L_2= \begin{pmatrix}  0.0594 & -1.98 \\ 0.495 & 0.01547 \end{pmatrix}.
\]

\begin{figure}[htb]
\begin{center}
\vskip -8mm
\includegraphics[width=8cm, keepaspectratio]{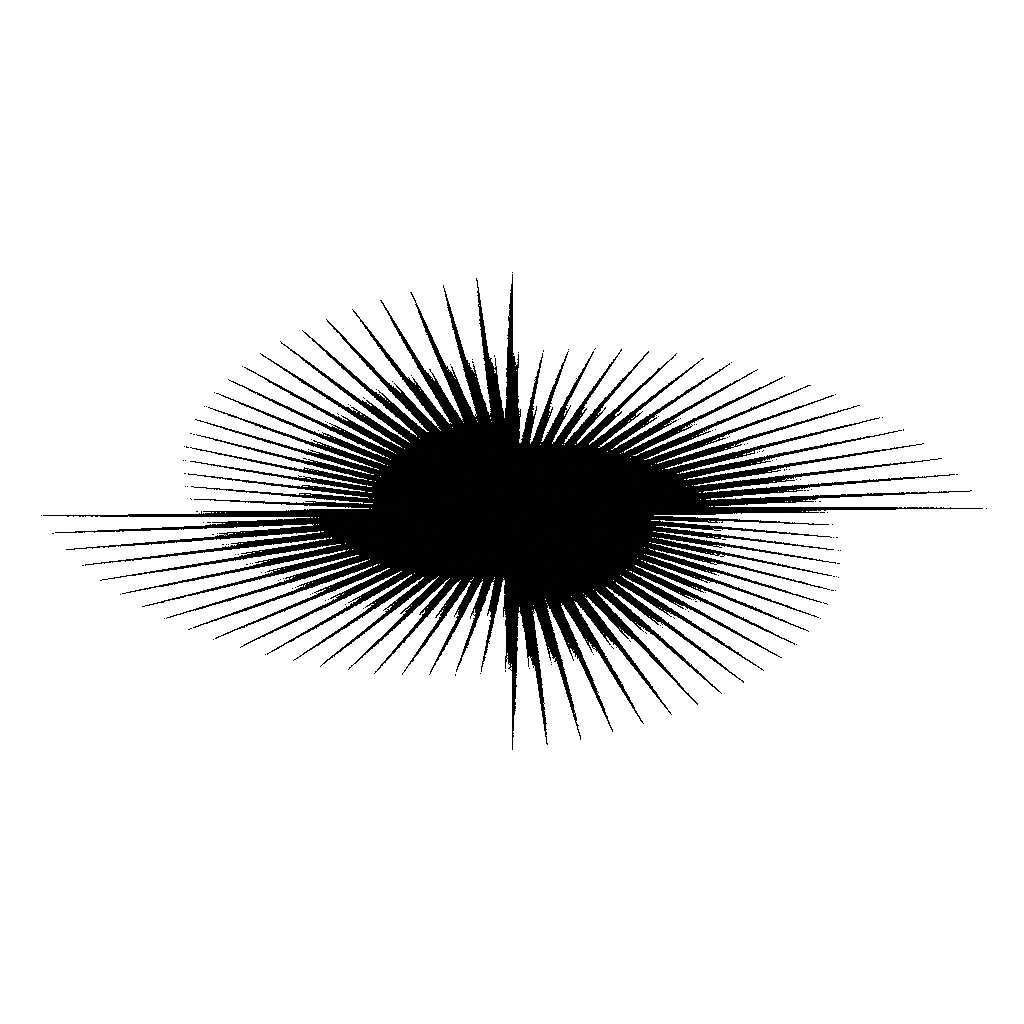}
\vskip -20mm
\caption{A transition attractor for a linear one-parameter family.}
 \label{fig:f3} 
\end{center}
\end{figure}

\section{Lower Transition Attractors, Upper Transition Attractors, and Semiattractors} \label{sec:uls}

Consider a one-parameter family 
\[
	F_t := \{ f_{(i,t)} : 1 \leq i \leq N\},\;t\in[0,\infty),
\] 
consisting of Lipschitz maps defined on 
a complete metric space $(\x,d)$.  Let $t_0$ be the threshold for the existence of an attractor
as given in Definition~\ref{def:t}.  We say that $\widehat{t_0}$ is a {\bf threshold for contractivity} if $F_t$ is contractive for all $0< t< \widehat{t_0}$ and 
\[\widehat{t_0} = \sup \{t : F_t \; \text{is contractive}\}.\]
{\it We assume throughout this section that $F_t$ has a finite contractivity threshold.}  Note that for 
$t<\widehat{t_0}$, the IFS $F_t$ has an attractor; hence \[\widehat{t_0} \leq t_0\] if a threshold $t_0$ exists. It is often the case and 
it is the interesting case when  $\widehat{t_0} = t_0$, but we know from the 
examples in Section~\ref{sec:ca} that this is not always true. Even when it is not
the case, the theorems in this section hold.  

It can be assumed without loss of generality that $\widehat{t_0}=1$. 
Indeed, we can redefine
\[\widetilde{f_{(i,t)}} := f_{(i,t t_0)} \]
and get $\widehat{t_0} =1$ for 
\[
	\widetilde{F_t} := \{ \widetilde{f_{(i,t)}} : 1 \leq i \leq N\}.
\]  
Therefore we restrict the parameter $t$ to the closed interval $[0,1]$ in this section and 
Section~\ref{sec:ut}.   
\vskip 2mm

The following conditions on the one-parameter family $F_t$ on the metric space $(\x,d)$
appear in the hypotheses of the results in this section.
\begin{enumerate}
\item[(H1)]  the map $t \mapsto f_{(i,t)}(x)\in\x$ is 
continuous for every $x\in\x$ and every $i=1,...,N$; 
\vskip 1mm
\item[(H2)] $Lip(F_t,d)  <1$ for all $t\in[0,1)$; 
\vskip 1mm
\item[(H3)] $q_i := \lim_{t\to 1^{-}} q_{i,t}$ 
exist for each $1\leq i\leq N$,
where $q_{i,t}$ denotes the unique fixed point of 
$f_{(i,t)}, \, t\in [0,1).$
Define $Q := \{q_i: 1\leq i\leq N\}$. 
\end{enumerate}

\vskip 2mm

\begin{remark} If the contractivity threshold for $F_t$ is $\widehat{t_0} =1$, then for every $t<1$ there exists admissible an metric $d_t$ such $Lip(F_t,d_t) <1$.  The somewhat stronger assumption (H2) states that there is a single metric $d$ such that $Lip(F_t,d)<1$.
\end{remark}

\begin{remark}\label{rem:H123imply}  

If (H1) and (H2) hold, then 
\begin{itemize}
\item[(a)] 
Each $F_t:\K(\x)\to\K(\x)$, $t<1$, is a Banach contraction 
in the Hausdorff metric.
\item[(b)] 
$Lip(F_1,d)\leq 1$.
In particular, the Hutchinson operator $F_1:\K(\x)\to\K(\x)$ 
is nonexpansive in the Hausdorff metric. 
\item[(c)]
The limit point $q_i \in Q$ in (H3) is a fixed point of $f_{(i,1)}$.
One should be aware, however, that $q_i$ is not necessarily 
a unique fixed point of $f_{(i,1)}$
(just think of the affine one-parameter family $f_{(i,t)}(x)=tx$;
 see also Example \ref{ex:dream3d}).
\end{itemize}
\end{remark}

\subsection{The Lower Transition Attractor and Semiattractor}

\begin{definition} \label{def:l}
The \textbf{lower transition attractor} of $F_t$ 
is the smallest (with respect to inclusion) set $A_{\bullet}$ 
which is $\mathbf{(F_{1},Q)}$-{\bf invariant}, i.e., 
$F_1(A_{\bullet})=A_{\bullet}$ and $A_{\bullet}\supseteq Q$.
(Equivalently, $A_{\bullet}$ is the smallest set with
$F_1(A_{\bullet}) \cup Q = A_{\bullet}$; see the first part 
of proof of Theorem~\ref{thm:lowA}.)
\end{definition}

\begin{definition}[\cite{lm1, MS}] \label{def:s}
Let $F$ be an IFS on a metric space $\x$. 
If the intersection is nonempty, then
the \textbf{semiattractor} of $F$ is 
\[ 
A_* := \bigcap_{x\in \x} Li(F^{(n)}(\{x\})),
\]
where  $Li(S_n)$ is the lower Kuratowski limit 
(\cite{HuPapa}) of a sequence of sets $S_n\subseteq\x$, 
i.e.,
\[
   Li(S_n) := \{ y \in \x: \text{there exist points}\quad 
   x_{n}\in S_{n} \quad\text{such that} 
   \quad x_{n} \rightarrow y\}.
\]
\end{definition}
\noindent Note that a semiattractor can be unbounded, e.g., \cite{lm1}.
The following properties of an IFS with semiattrator $A_*$ hold.

\begin{theorem}[\cite{MS}] \label{thm:semi}
If $F$ is an IFS on a complete metric space with semiattractor $A_*$, then
\begin{enumerate}
\item $F(A_*)= A_*$; moreover $A_*$ is the smallest $F$-invariant set.
\item If $F$ admits an attractor $A$ with a full basin $\x$,
 then $A_* = A$.
\end{enumerate}
\end{theorem}

The notion of a semiattractor comes into play in \cite{KS,S}, 
where  functions that are not contrations are added 
to an IFS consisting of contractions.
This allows for the use of standard methods 
for computer drawing of the attractor of contractive IFS.
\vskip 2mm

Theorem~\ref{thm:lowA} below is a significantly more general version of those parts of \cite[Theorem 8.2] {V} pertaining to the lower transition attractor.
In addition, part (iv) of Theorem~\ref{thm:lowA}  relates the lower transition attractor of one-parameter family $F_t$ to the semiattractor of an associated IFS.  

\begin{theorem} \label{thm:lowA} Let $F_t$ be a one-parameter
family $F_t, \; t \in [0,1],$ on a complete metric space $(\x,d)$ that satisfies
(H1), (H2) and  (H3).
Then the lower transition attractor $A_{\bullet}$ always exists. Moreover
$A_{\bullet}$ obeys the following properties:
\vskip 1mm
\begin{enumerate}
\item[(i)] $A_{\bullet} = \bigcap \{A\in 2^{\x}:
F_1(A)=A \quad\text{and}\quad Q\subseteq A\}$;
\vskip 1mm
\item[(ii)] $A_{\bullet} = 
\overline{\bigcup_{n\geq 0} F_1^n(Q)}$;
\vskip 1mm
\item[(iii)] 
$A_{\bullet} = 
\overline{\bigcup_{n\geq 0} F_1^n(Q')}$,
where $Q' = \{ q_i : i \in J\}$ and $J\neq\emptyset$ is such that 
$\{i\in\{1,...,N\} : Lip(f_{(i,1)}) = 1\} 
\subseteq J \subseteq \{1,...,N\}$. 
In other words, $Q' \subseteq Q$  contains at least 
the fixed point limits of those functions $f_{(i,1)}$ 
that are not a contraction. 
\vskip 1mm
\item[(iv)] 
The lower transition attractor $A_{\bullet}$ is 
the semiattractor of any IFS of the form 
$F_1^{\flat} := F_1 \cup \{\check{q}(x): q\in Q'\}$,
where $\check{q}(x) := q$ is the constant map on $\x$.
\end{enumerate}
\end{theorem} 
\begin{proof}
Clearly, $F_1^\flat(S)=F_1(S)\cup Q'$ for any nonempty 
$S\subseteq \x$, and $F_1(Q')\supseteq Q'$.
First note that the set $A$ is the smallest 
$F_1^{\flat}$-invariant set if and only if 
$A$ is the smallest $F_1$-invariant set which contains $Q'$.
Indeed, $A= F_1^{\flat}(A) = F_1(A)\cup Q'$ implies 
$F_1(A)\subseteq A$ and $A=A\cup Q'\supseteq Q'$.
Hence
\[
F_1(A)= F_1(A\cup Q') = F_1(A) \cup F_1(Q') 
\supseteq F_1(A)\cup Q'=A.
\]
In the reverse direction, 
if $A=F_1(A)$ and $A\supseteq Q'$, then 
$F_1^{\flat}(A) = F_1(A) \cup Q' = A\cup Q' = A$.

Second, observe that the subsystem 
$\{\check{q}: q\in Q'\} {\subseteq} F_1^{\flat}$ consists of contractions and admits a semiattractor (even attractor),
which is $Q'$. Hence, by the Lasota--Myjak criterion
(\cite{MS} Theorem 6.3), $F_1^{\flat}$ admits 
a semiattractor, denoted $A_{*}'$.
Furthermore, since $A_{*}' \supseteq Q'$ and 
$(F_1^{\flat})^n(Q') = F_1^n(Q')$, we have
$A_{*}'= \overline{\bigcup_{n\geq 0} F_1^n(Q')}$
due to the self-regeneration formula in the 
Lasota--Myjak criterion (\cite{MS} Theorem 6.3 eq. (6.9)).
In particular, the above is true for $Q'=Q$,
in which case we write $A_{*}$ for the semiattractor.

We have established the existence of 
a lower transtition attractor, which is $A_{\bullet}=A_{*}$, 
and properties (i) and (ii).

Third, we shall establish that all $A_{*}'$ are equal to $A_{*}$.
This will give the representation of $A_{\bullet}$ 
as a semiattractor of any $F_1^{\flat}$, and in turn property (iii).
Of course $A_{*}'{\subseteq} A_{*}$. 
Consider $q_i=f_{(i,1)}(q_i)$ with $i\not\in J$.
Since $\{q_i\}$ is the attractor of the subsystem 
$\{f_{(i,1)}\}\subseteq F_1^\flat$, we have $q_i\in A_*'$.
Overall $Q{\subseteq} A_{*}'$ and 
$A_{*}\subseteq A_{*}'$.
\end{proof} 

Under mild additional conditions on $F_t$, 
the lower transition attractor is compact.  See Corollary~\ref{cor:compact} and 
Remark~\ref{rem:monoid} below. These results require extending some
 concepts defined in Section~\ref{sec:intro} to infinite IFSs, e.g., \cite{Mantica, MauldinUrbanski}.  
Let $F$ be a finite or infinite IFS on a complete metric space $\x$. 
The \textbf{Hutchinson operator} on $\x$ 
induced by $F$ is the operator $F:2^{\x}\to 2^{\x}$
acting on the power set of $\x$ and given by the formula
\[F(S) := \overline{\bigcup_{f\in F} f(S)}\]
for all $S \subseteq \x$.
Note that, for a finite IFS, the closure can be dropped if $S$ is compact.
An IFS $F$ on $\x$ will be called 
{\bf compact} if $F(K)$ is compact
for every compact set $K\subseteq \x$.
Clearly, any finite IFS is compact.

Given an IFS $F$ on $\x$, the {\bf monoid}
induced by $F$ is
\begin{equation*}
  \mathbb{M}(F) := \{f_1\circ \dots\circ f_k : 
  f_1,\dots,f_k\in F, k\in\mathbb{N}\} \cup\{\id_{\x}\}.
\end{equation*}
A monoid can be treated as a new IFS. In particular, we may speak of a compact monoid.

\begin{cor} \label{cor:compact}
Let $F_t$ be as in Theorem \ref{thm:lowA} and let
 $J = \{1\leq i\leq N: Lip(f_{(i,1)}, d)= 1\}$. If  the monoid $\mathbb{M}(\{f_{(i,1)}: i\in J\})$  
is compact, then the lower transition attractor $A_{\bullet}$ of $F_t$ is compact.
\end{cor}

\begin{proof}
The statement follows from Theorem \ref{thm:lowA} (iv) 
and from \cite[Theorems 4.1]{S}. 
\end{proof}

\begin{remark} \label{rem:monoid} 
If either of the following two conditions hold, then the monoid 
$\mathbb{M}(\{f_{(i,1)}: i\in J\})$ 
 is compact.
\begin{itemize}
\item
$J=\{i_{*}\}$ for some $i_{*}\in\{1,...,N\}$,
and $f_{(i_{*},1)}$ is a periodic isometry, 
cf. \cite{KS};
\item
$\x$ is proper and
all $f_{(i,1)}$, $i\in J$, have a common fixed 
point (not necessarily unique),
cf. \cite[Theorem 4.2 (ii), Lemma 2.2 item 3]{S}.  
 \end{itemize}
\end{remark}

 The compactness of the lower transition attractor $A_{\bullet}$ in
Corollary~\ref{cor:compact} cannot be inferred from (H1), (H2), and (H3) alone. 
Example~\ref{ex:five} below is a counterexample. 

\begin{example}  \label{ex:five} [A one-parameter family satisfying (H1), (H2), and (H3)
whose lower transition attractor is not compact.]

 On $\mathbb R$ let $F_t := \{g_t, f_{t} \}$, where $g_t(x) = -tx$  and $f_t(x) = -tx+t +1$.   
For $t\in (0,1)$ we have  $A_t = [-t/(1-t), 1/(1-t)]$.  In this case $A_{\bullet} = \R$.  
\end{example} 

Example~\ref{ex:dream3d} below is a $3$-dimensional example illustrating the previous results in
this section.

\begin{example} \label{ex:dream3d} 

In $\R^3$ let $F_t = \{f_{(i,t)}, 1\leq i\leq 5\}$ be the one-parameter 
affine family where $f_{(i,t)}(v) = t\, L_i(v-q_i) +q_i$, and
\[
L_1=L_2=L_3 = 
\begin{pmatrix} 
0.5 & 0 & 0 \\ 
0 & 0.5 & 0 \\ 
0 & 0 & 0.5 \\
\end{pmatrix}, 
\qquad 
L_4 = 
\begin{pmatrix} 
0 & 0 & 1 \\ 
0 & 1 & 0 \\ 
-1 & 0 & 0 \\
\end{pmatrix}, \qquad  L_5 = 
\begin{pmatrix} 
1 & 0 & 0 \\ 
0 & -1 & 0 \\ 
0 & 0 & 1 \\
\end{pmatrix}.
\]
The map $L_4$ is the rotation by $\pi/2$ about $y$-axis, and
$L_5$ is the reflection in the $xz$-plane.  The fixed points are
\[ q_i = \Big (\cos\frac{2\pi(i-1)}{3}, \, \sin\frac{2\pi(i-1)}{3}, \, 0 \Big )\mbox{ for $i=1,2,3$,} \qquad q_4=(0,1,0), \qquad 
q_5=(0,0,1),\]
where $q_1, q_2, q_3$ are the third roots of unity in the $xy$-plane.
Note that the attractor of the IFS $\{f_{(i,1)}, 1\leq i\leq 3\}$
is the Sierpi\'{n}ski triangle in the $xy$-plane
with vertices $q_1, q_2, q_3$.

For each $1\leq  i\leq 5$,  the point $q_i$ is a common fixed point of $f_{(i,t)}$ for $t\in[0,1]$. 
However, $q_i$ is not the only fixed point of $f_{(i,1)}$ for $i=4,5$. More precisely,
$f_{(4,1)}$ has the whole $y$-axis as its
set of fixed points; $f_{(5,1)}$ has the whole $xz$-plane as its set of fixed points; 
and $(0,0,0)\neq q_4, q_5$ is the only common fixed point
of $f_{(4,1)}$ and $f_{(5,1)}$.

On the left in Figure~\ref{fig:dream3d} is the attractor $A_t$ of $F_t$ 
for $t=0.8$.  By Theorem~\ref{thm:lowA} the lower transition 
attractor $A_{\bullet}$ for IFS family $F_t$ of Example~\ref{ex:dream3d} exists; it appears
on the right in  Figure~\ref{fig:dream3d}.  By Corollary~\ref{cor:compact} $A_{\bullet}$ is
compact, the relevant monoid being finite.
Figure~\ref{fig:dream3d} was generated using Mekhontsev's IFStile program \cite{IFStile}.
To draw $A_{\bullet}$ using this program we have applied part (iv) of
Theorem \ref{thm:lowA} which identifies $A_{\bullet}$ as
a semiattractor of a suitable IFS 
$F_1^{\flat} := \{f_{(1,1)},f_{(2,1)},f_{(3,1)},f_{(4,1)},f_{(5,1)},\check{q}_4,\check{q}_5\}$ 
related to $F_t$. Then the resulting IFS $F_1^{\flat}$ 
was replaced with a contractive IFS according to 
\cite[Theorem 4.1 (B)]{S}. 
\end{example} 
\vskip 2mm

\begin{figure}[htb]  
\vskip 3mm
\includegraphics[width=5cm, keepaspectratio]{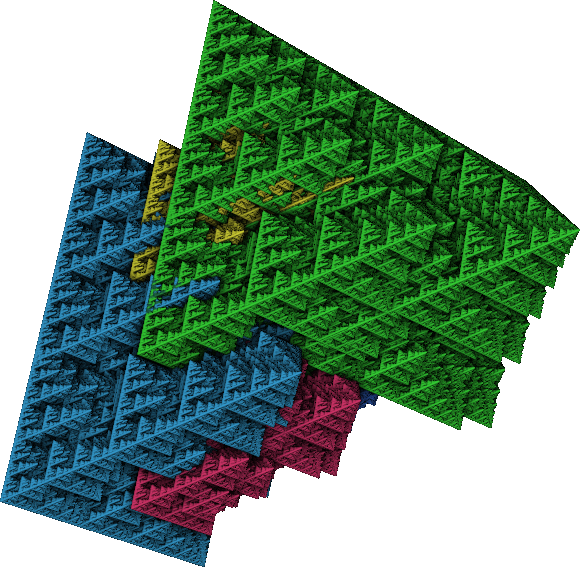} 
\hskip 10mm
\includegraphics[width=5cm, keepaspectratio]{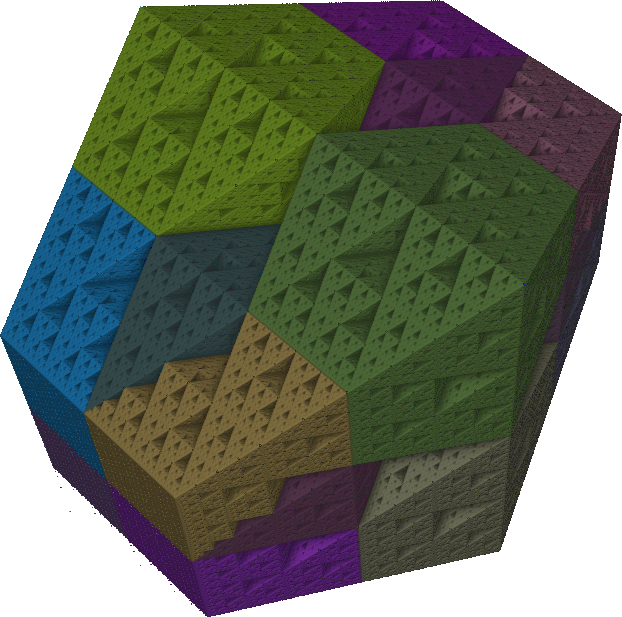}
\caption{The attractor $A_t$ for the one-parameter 
affine family $F_t$ of Example~\ref{ex:dream3d} 
for parameter value $t = .8$ 
and the lower transition attractor $A_{\bullet}$ of { $F_t$.}}
\label{fig:dream3d}
\end{figure}

\subsection{The Upper Transition Attractor}

\begin{definition} \label{def:u}  
Call a compact set $A^{\bullet}$ an {\bf upper transition attractor} 
of a one-parameter IFS family $F_t := \{f_{(1,t)}, f_{(2,t)} \dots, f_{(N,t)}\},
 \,t\in [0,1]$, if there is an increasing sequence $t_n \rightarrow 1$ such that 
\[
A^{\bullet} = \lim_{n\rightarrow \infty} A_{t_n}. 
\]
\end{definition} 

  Theorem \ref{thm:ut2}, Theorem \ref{thm:LU}, and 
Proposition~\ref{prop:sym} below are strong versions of  results
on upper transition attractors
and their relation to the lower transition attractor that were proved in \cite{V} 
 only for special cases of  one-parameter similarity families.

\begin{theorem} \label{thm:ut2}	
Assume that $\x$ is proper, $F_t$ satisfies (H2)
and, for each $1\leq i\leq N$, the map 
$[0,1]\ni t\mapsto f_{i,t}\in C(\x)$ is continuous 
 with respect to the topology of uniform convergence in $C(\x)$. 
Then $F_t$ admits at least one upper transition attractor.
\end{theorem}

To prove this theorem we need the following lemma.

\begin{lemma} \label{lem:seq}
Assume that $(f_n)$ is a sequence of contractions
on a complete metric space $(\x,d)$, uniformly convergent 
to some function $f$. Then the set of fixed points 
of maps $f_n$, $n\in\N$, is bounded.
\end{lemma}

\begin{proof}
 Let $d_{\sup}(f,g):=\sup_{x\in \x}d(f(x),g(x))$ for $f,g:\x\to\x$. For $n\in\N$, let $x_n$ be the fixed point of $f_n$.
Fix an $n_0\in\N$ such that $d_{\sup}(f_n,f)<1$ for all $n\geq n_0$. 
For every $n\geq n_0$, we have
\[\begin{aligned}
d(x_n,x_{n_0}) &\leq d(f_{n}(x_n),f_{n_0}(x_n))+
d(f_{n_0}(x_n),f_{n_0}(x_{n_0})) \\ 
&\leq
d_{\sup}(f_n,f_{n_0})+\on{Lip}(f_{n_0})d(x_n,x_{n_0}) \\ &\leq 
d_{\sup}(f_n,f)+d_{\sup}(f,f_{n_0})+\on{Lip}(f_{n_0})d(x_n,x_{n_0}).
\end{aligned}\]
Hence
\[
d(x_n,x_{n_0})\leq 
\frac{d_{\sup}(f_n,f)+d_{\sup}(f,f_{n_0})}{1-\on{Lip}(f_{n_0})}
\leq \frac{2}{1-\on{Lip}(f_{n_0})}.
\]
Therefore
\[
\on{diam}\{x_n:n\in\N\}\leq 2\max \Big\{ d(x_1,x_{n_0}),..., 
d(x_{n_0-1},x_{n_0}),\frac{2}{1-\on{Lip}(f_{n_0})}\Big\}<\infty.
\]
\end{proof}

\begin{proof}[Proof of Theorem \ref{thm:ut2}] 
Suppose that the assertion does not hold. Then we can find 
a convergent sequence $(t_n)\subset[0,1]$ so that 
the family $A_{t_n}$, $n\in\N$, of attractors of $F_{t_n}$, 
are not all included in some bounded set. In other words, 
the set $\{A_{t_n}:n\in\N\}$ is not bounded in $\K(X)$. 
Let $t=\lim_{n\to\infty}t_n$. Now observe that 
for every compact set $K\in\K(X)$, we have
\[\begin{aligned}
h({F_{t_n}}(K),{F_{t}}(K))&= 
h\Big(\bigcup_{i=1}^{N} f_{(i,t_n)}(K),
\bigcup_{i=1}^{N} f_{(i,t)}(K)\Big) \\ &
\leq \max\{h(f_{(i,t_n)}(K),f_{(i,t)}(K)): i=1,...,N\} \\ &
\leq \max\{d_{\sup}(f_{(i,t_n)},f_i):i=1,...,N\}
\end{aligned}\]
Hence
\[
 \sup\{h({F_{t_n}}(K),{F_{t}}(K)):K\in\K(X)\}
\leq \max\{d_{\sup}(f_{(i,t_n)},f_i):i=1,...,N\}\to 0.
\]
Therefore the assumptions of  Lemma~\ref{lem:seq} are satisfied 
(for a family of Hutchinson operators) and the family 
$\{A_{\F_{t_n}}:n\in\N\}$ is bounded in $\K(X)$, a contradiction.
\end{proof}

 The existence of an upper transition attractor in Theorem~\ref{thm:ut2} cannot be inferred from (H1), (H2), and (H3) alone; see Example~\ref{ex:five}.  Neither can it be inferred from (H1) and the assumption that all maps in $F_t$ are contractions
 for all $t\in [0,1]$; see Example~\ref{ex:nadler} below.

\begin{example}\label{ex:nadler}
Motivated by the construction in \cite[Example 1]{Nadler}, 
we will construct a one-parameter family of IFSs $F_t$, $t\in[0,1]$,
with the following properties:
\begin{enumerate}
\item[(a)] $F_t$ satisfies (H1);
\item[(b)]  for all $t\in [0,1]$ all maps in $F_t$ are contractions, in particular
$F_t$ satisfies (H2);
\item[(c)] $F_t$ has no upper transition attractor.
\end{enumerate} 

Let $\ell^1$ be the
Banach space of absolutely convergent sequences of real numbers.  We will construct a
function ${f}_{t}:\ell_1\to\ell_1$ such that the one-parameter family $F_t  =\{f_t\}$,
consisting of a single function, will satisfy the properties (a), (b), (c) above. 
For each $t\in[0,1]$ the function $f_t$ will have the form
\[
{f}_{t}(x) = \begin{cases} t\frac{{\phi}_{t}(x)}{{\phi}_{t}(z_t)} \cdot z_t+(1-t)z_t \quad \text{if} \;\;  t<1 \\ 
\mathbf{0} \quad \text{if} \; \; t=1,
\end{cases}\]
where $\mathbf{0}$ is the sequence of zeros and the linear functional ${\phi}_{t}:\ell_1\to\R$, $t\in[0,1]$,
has the form
\begin{equation*} 
{\phi}_{t}(x) =\begin{cases} \alpha(t)x_{n_t}+\beta(t)x_{n_t+1} \quad \text{if} \; \; t<1 \\
0 \quad \text{if} \; \; t=1, \end{cases}
\end{equation*}
where we use the notation $x=(x_n) \in \ell^1$.  It remains to define 
$z_t \in \ell^1$ for each $t\in[0,1]$, the functions $\alpha,\beta:[0,1)\to[0,1]$, and the integer $n_t$ for all
$t<1$, and to show that properties (a), (b), (c) hold for $F_t$.

To define $\alpha,\beta$ and $n_t$, choose any increasing sequence $(a_n)$ of real numbers 
tending to $1$ and such that $a_0=0$. For each $t\in[0,1)$, find 
$n_t\in\{0,1,2,...\}$ so that $a_{n_t}\leq t<a_{n_t+1}$. Clearly, for any $t\in[0,1)$, we have that $t\in [a_n,a_{n+1})$ if and only if $n_t=n$. Now choose maps $\alpha,\beta,c:[0,1)\to[0,1]$ 
 which satisfy the following conditions:
\begin{itemize}
\item[(i)] $\alpha,\beta,c$ are right continuous on $[0,1)$;
\item[(ii)] $\alpha,\beta,c$ are continuous on each interval $(a_n,a_{n+1})$, $n\in\N\cup\{0\}$;
\item[(iii)] for any $n\in\N\cup\{0\}$, we have that
\begin{itemize}
\item[(iiia)] $\alpha(a_n)=1$ and $\lim_{t\to a_{n+1}^-}\alpha(t)=0$;
\item[(iiib)] $\beta(a_n)=0$ and $\lim_{t\to a_{n+1}^-}\beta(t)=1$;
\item[(iiic)] $c(a_n)=1$ and $\lim_{t\to a_{n+1}^-}c(t)=0$;
\end{itemize}  
\item[(iv)] $\max\{\alpha(t),\beta(t)\}=1$ for all $t\in[0,1)$;
\item[(v)] $c(t)\alpha(t)+(1-c(t))\beta(t)>t$ for all $t\in[0,1)$.
\end{itemize}
The choice of maps $\alpha,\beta,c$ is possible. For example, $\alpha$ can be constant $1$ on each interval $[a_n,\frac{1}{2}(a_n+a_{n+1})]$ and affine on $[\frac{1}{2}(a_n+a_{n+1}),a_{n+1})$. Similarly $\beta$ can be affine on each interval $[a_n,\frac{1}{2}(a_n+a_{n+1})]$ and constant $1$ on $[\frac{1}{2}(a_n+a_{n+1}),a_{n+1}))$. Finally, $c$ can be constant $1$ on $[a_n,\frac{1}{2}(a_n+a_{n+1})-\xi]$, constant $0$ on $[\frac{1}{2}(a_n+a_{n+1})+\xi,a_{n+1})$ and affine on $[\frac{1}{2}(a_n+a_{n+1})-\xi,\frac{1}{2}(a_n+a_{n+1})+\xi]$, where $\xi>0$ is sufficiently small (for example, $\xi=\frac{1}{2}(a_{n+1}-a_n)(1-a_{n+1})$).   Graphs of $\alpha,\beta$ and $c$ are illustrated in Figure~\ref{abc}.

\begin{figure}[htb]  
\vskip 3mm
\includegraphics[width=7cm, keepaspectratio]{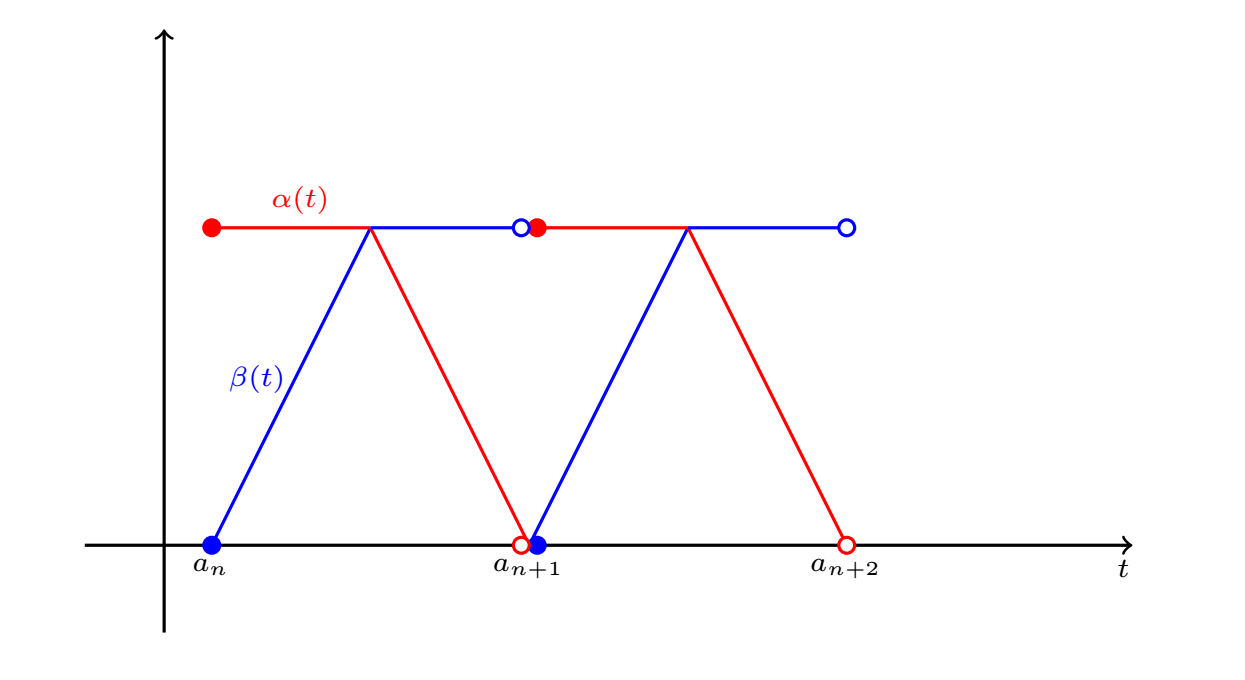} 
\hskip -10mm
\includegraphics[width=7cm, keepaspectratio]{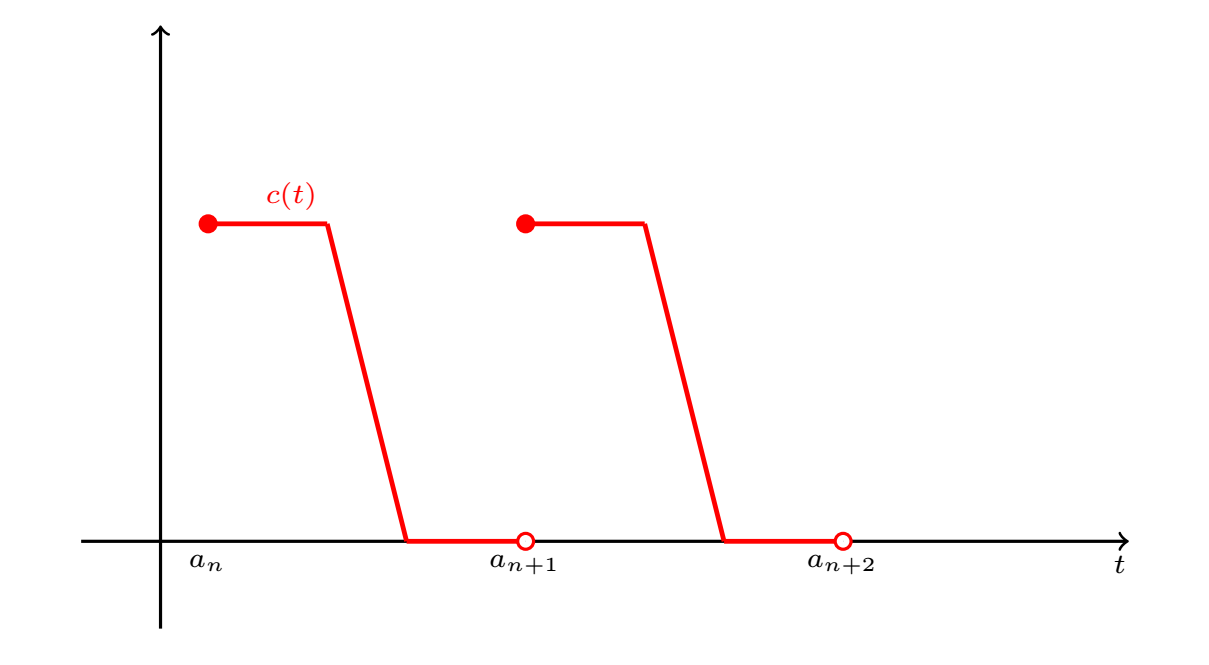}
\caption{The graphs of $\alpha,\beta$ and $c$}
\label{abc}
\end{figure}

To define $z_t$, use (iv) and the classical correspondence between linear functionals on $\ell_1$ and the space $\ell_\infty$ to obtain
\begin{equation}\label{filjuly1}
||{\phi}_{t}||=  \Big|\Big|\Big(0,...,0,\alpha(t),\beta(t),
0,...\Big)\Big|\Big|_\infty=\max\{|\alpha(t)|,|\beta(t)|\}=1
\end{equation}
for all $t\in[0,1)$.
Now fix any $x\in\ell_1$ and define the map 
$g_x:[0,1]\to \R$ by
\[
g_x(t):={\phi}_{t}(x).
\]
Next we show that $g_x$ is continuous. By (i) and (ii) we see that $g_x$ is right continuous on the whole interval $[0,1)$ and continuous on each interval $(a_n,a_{n+1})$, $n\in\N\cup\{0\}$. Using (iiia) and (iiib), for $n\geq 1$ we have
\[
\lim_{t\to a_n^{-}}g_x(t)= 
\lim_{t\to a_n^{-}} 
\Big(\alpha(t)x_{n-1}+\beta(t)x_{n}\Big)=x_n=g_x(a_n),
\] 
which gives left continuity at $a_n$ and, in consequence, its continuity at $a_n$. Finally, we observe that $g_x$ is continuous at $1$:
\[
0\leq \lim_{t\to 1}|g_x(t)|\leq 
\lim_{n\to\infty}(|x_n|+|x_{n+1}|)=0={\phi}_{1}(x).
\]
For $t\in[0,1)$ define 
\begin{eqnarray*}
z_t:=c(t)e_{n_t}+(1-c(t))e_{n_t+1}\in\ell_1 
\end{eqnarray*}
where $e_n$ is the $n$-th unit vector in $\ell_1$.
Note that, by (v), we have that
\begin{equation}\label{filmay1}
{\phi}_{t}(z_t)=\alpha(t)c(t)+(1-c(t))\beta(t)>t.
\end{equation}

We now verify statement (a), that 
\[
[0,1]\ni t\mapsto {f}_{t}(x)
\]
is continuous.
As was shown for $g_x$, using (i), (ii) and (iii) we see that the map
\[
[0,1)\ni t\to \phi_t(z_t)\in\R
\]
is continuous.
The continuity of $[0,1)\ni t\mapsto {f}_{t}(x)$ 
follows easily from the continuity of the maps
$t\mapsto {\phi}_{t}(x)$ and $t\mapsto \phi_t(z_t)$ for $t\in[0,1)$, which were observed earlier, and the continuity of the map $[0,1)\ni t\mapsto z_t\in\ell_1$ that can be proved in a similar way.
Furthermore, since ${\phi}_{t}(x)\to {\phi}_{1}(x)=0$ and 
$(1-t)\to 0$ when $t\to 1$, we have
\[
||{f}_{t}(x)||\leq 
\frac{t}{{\phi}_{t}(z_t)}\cdot|{\phi}_{t}(x)|||z_t||+
(1-t)||z_t||\leq 
|{\phi}_{t}(x)|+(1-t)\to 0,
\]
when $t\to 1$. Hence the map $[0,1]\ni t\mapsto {f}_{t}(x)$ is also continuous at $1$. 

We next verify statement (b), that
 ${f}_{t}$ is {a contraction} for $t<1$.   We have
\[
||{f}_{t}(x)-{f}_{t}(y)||=
\frac{t}{{\phi}_{t}(z_t)}|{\phi}_{t}(x)-{\phi}_{t}(y)|
\cdot ||z_t||\leq
\frac{t}{{\phi}_{t}(z_t)}||{\phi}_{t}||\cdot||x-y||.
\]
Moreover $\frac{t}{{\phi}_{t}(z_t)}<1$ from \eqref{filmay1} and $\|{\phi}_{t}\|=1$ from \eqref{filjuly1}.

It only remians to check property (c), that $F_t$ has no upper transition attractor.
We have
\[
{f}_{t}(z_t)= t\frac{{\phi}_{t}(z_t)}{{\phi}_{t}(z_t)}z_t
+(1-t)z_t=z_t
\]
for $t\in[0,1)$.
Therefore $z_t$ is a unique fixed point of ${f}_{t}$; in particular, $A_{\F_t}=\{z_t\}$. 
On the other hand, the unique fixed point of ${f}_{1}$ 
is clearly the zero sequence $\mathbf{0}$, so $\{\mathbf{0}\}$ is the only candidate for an upper transition attractor of $F_t$ (see Theorem \ref{thm:LU}(i) below). However, for $t<1$ we have
\[
  ||z_t-\mathbf{0}||=||z_t||=1.
\]
Thus all three properties (a),(b),(c) of our example have been verified.
\end{example}

For Theorem~\ref{thm:LU} below and in Section \ref{sec:ut} we will need the following
technical lemma.

\begin{lemma}\label{fildec5}
Let $\x$ be a metric space and let {$f_t,\;t\in[0,1]$, 
be a family of nonexpansive selfmaps of $\x$ 
such that for every $x\in\x$, the map
$[0,1]\ni t\mapsto f_t(x)$ is continuous. Then 
for every nonempty and compact set} $D\subseteq \x$,
\begin{equation*}
\forall_{\varepsilon>0}\; \exists_{\delta>0}\; 
\forall_{s,t\in[0,1]}\;\; (|s-t|<\delta\;\Rightarrow\;
\sup_{x\in D}\;\; d(f_s(x),f_t(x))\leq\varepsilon).
\end{equation*}
In particular, for every nonempty and 
compact set $D\subseteq \x$, the map 
\[
[0,1]\ni t\mapsto f_t(D)\in\K(\x)
\]
is uniformly continuous.
\end{lemma}

\begin{proof}
Assume first that the set $D$ is finite. Fix $\varepsilon>0$. 
Then for every $t\in[0,1]$, we can find $\delta_t>0$ such that 
for every $s\in[0,1]$ with $|s-t|<\delta_t$ we have
\begin{equation}\label{filaa1}
\sup_{x\in D}\;d(f_t(x),f_s(x))<\frac{\varepsilon}{2}.
\end{equation}
The choice of $\delta_t$ is possible since $D$ is finite and 
the map $t\mapsto f_t(x)$ is continuous for every $x\in D$.\\
Since $[0,1]$ is compact, we can choose a finite subcover of the open cover 
$\Big(t-\frac{\delta_t}{2},t+\frac{\delta_t}{2}\Big)$, $t\in[0,1]$. 
Let $\Big(t_i-\frac{\delta_{t_i}}{2},t_i+\frac{\delta_{t_i}}{2}\Big)$, 
$i=1,...,k,$ be this subcover and choose
\[
\delta :=\frac{1}{2}\min\{\delta_{t_i}:i=1,...,k\}.
\]
Now let $s,t\in[0,1]$ be such that $|s-t|<\delta$. By the choice of 
$t_1,...,t_k$, we can find $i=1,...,k$ so that 
\[
|t-t_i|<\frac{\delta_{t_i}}{2}.
\]
Then also
\[
|s-t_i|\leq|s-t|+|t-t_i|<\delta+\frac{\delta_{t_i}}{2}\leq \delta_{t_i}.
\] 
Hence by (\ref{filaa1}), for every $x\in D$, we have
\[
d(f_s(x),f_t(x))\leq d(f_s(x),f_{t_i}(x))+d(f_{t_i}(x),f_t(x))
<\frac{\varepsilon}{2}+\frac{\varepsilon}{2}=\varepsilon
\]
and
\[
\sup_{x\in D}d(f_s(x),f_t(x))\leq\varepsilon.
\]
Now assume that $D$ is nonempty and compact. 
Take any $\varepsilon>0$, and find a finite set 
$D' \subseteq D$ so that the Hausdorff distance 
$h(D',D)<\frac{\varepsilon}{3}$. By previous considerations, 
there exists $\delta>0$ such that if $|s-t|<\delta$, then
\[
\sup_{x\in D'}\;\;d(f_s(x),f_t(x))\leq\frac{\varepsilon}{3}.
\]
If $x\in D$, then we can find $x'\in D'$ so that 
$d(x,x')<\frac{\varepsilon}{3}$, and thus
\[
d(f_t(x),f_s(x))\leq d(f_t(x),f_t(x'))+d(f_t(x'),f_s(x'))
+d(f_s(x'),f_s(x))\leq 
2d(x,x')+\frac{\varepsilon}{3}\leq\varepsilon.
\]
Therefore
\[
\sup_{x\in D}\;\;d(f_s(x),f_t(x))\leq{\varepsilon}.
\]
\end{proof}

In what follows, we denote the Hausdorff distance by $h$.

\begin{theorem} \label{thm:LU}
Let $F_t$, $t\in[0,1]$, satisfy (H1) and (H2).
If $A^{\bullet}$ is any upper transition attractor of $F_t$, then 
\begin{enumerate}
\item[(i)] $F_1(A^{\bullet}) = A^{\bullet}$.
\end{enumerate}
If, in addition, $F_t$ satisfies (H3), then 
\begin{enumerate}
\item[(ii)] $A^{\bullet}\supseteq Q$;
in particular $A^{\bullet}\supseteq A_{\bullet}$,
\end{enumerate}
where $A_{\bullet}$ is the lower transition attractor of $F_t$
and $Q$ is the set of limit fixed points from (H3).
\end{theorem}

\begin{proof}
Let $t_n\to 1$ be such that $A_{t_n}\to A^{\bullet}$ 
with respect to $h$ as $n\to\infty$.
To establish (i) recall that each $F_{t_n}$ and $F_1$
are nonexpansive with respect to $h$ 
 (part (b) of Remark \ref{rem:H123imply}).
Furthermore, according to Lemma \ref{fildec5},
we have 
\begin{equation*}
h(F_1(A^{\bullet}), F_{t_n}(A^{\bullet}))
\leq \max_{1\leq i\leq N} 
h(f_{(i,1)}(A^{\bullet}), f_{(i,t_n)}(A^{\bullet}))
\to 0. 
\end{equation*}
Hence, by using $F_{t_n}(A_{t_n}) = A_{t_n}$
we get
\[\begin{aligned}
h(F_1(A^{\bullet}), A^{\bullet}) & \leq 
h(F_1(A^{\bullet}), F_{t_n}(A^{\bullet})) +
h(F_{t_n}(A^{\bullet}), F_{t_n}(A_{t_n})) +
h(F_{t_n}(A_{t_n}), A^{\bullet}) \\ & \leq 
h(F_1(A^{\bullet}), F_{t_n}(A^{\bullet})) +
2 h(A_{t_n}, A^{\bullet}) 
\underset{n\to\infty}{\longrightarrow} 0.
\end{aligned}\]

Now we establish (ii).
Observe that $q_{i,t_n}\in A_{t_n}\to A^{\bullet}$, 
and $q_{i,t_n}\to q_i\in Q$ as $n\to\infty$, 
Thus $Q\subseteq A^{\bullet}$.
Hence $A^{\bullet}$ is $(F_1,Q)$-invariant, and therefore
it contains $A_{\bullet}$.
\end{proof}

\begin{remark}
Assuming (H1), (H2) and (H3), $A_{\bullet}$ is compact
whenever $A^{\bullet}$ exists.
\end{remark}

\begin{proposition} \label{prop:sym}  Assume that  $F_t$, $t\in[0,1]$,
satisfies (H1) and (H2).  Let $f_{(i_*,1)}$ be an isometry
for some $i_{*}\in\{1,...,N\}$. 
\begin{enumerate}
\item[(a)]
If there exists an upper transition attractor $A^{\bullet}$,
then it is $f_{(i_*,1)}$-symmetric,
i.e., \linebreak $f_{(i_*,1)}(A^{\bullet}) =A^{\bullet}$.
\item[(b)]
If there exists a lower transition attractor $A_{\bullet}$ 
that is compact, then it is $f_{(i_*,1)}$-symmetric. 
\end{enumerate}
\end{proposition}
\begin{proof}
Observe that 
$f_{(i_*,1)}(A^{\bullet}) \subseteq A^{\bullet}$.
Then the isometry $f_{(i_*,1)}$ is surjective
on compactum $A^{\bullet}$. 
Analogously for $A_{\bullet}$.
\end{proof}

\section{The Existence of a Unique Upper Transition Attractor} \label{sec:ut}

This section addresses Question~\ref{conj:A} in the introduction. 
Theorem~\ref{thm:fildec15} below gives an affirmative answer 
for a large class of one-parameter IFS families. 

We start with Lemma~\ref{abc3} below for infinite IFSs, which is already known 
for finite IFSs.  Here the Hutchinson operator $F : 2^{\x} \rightarrow  2^{\x}$ is
as defined in Section~\ref{sec:uls}.

\begin{definition}
For a finite or infinite IFS $F$, a nonempty compact set $A$ is a \textbf{Hutchinson attractor}
on a complete metric space $\x$ if
\begin{itemize} 
\item (\textit{invariance}) $F(A) = A$, and
\item (\textit{attraction}) $A=\lim_{n\to\infty}F^{(n)}(S)$,
\end{itemize}
for every nonempty closed and bounded set $S\subseteq \x$,
the limit with respect to the Hausdorff metric. Note that a Hutchinson attractor, if it exists, is unique.
\end{definition}
A generalization of the Hutchinson theorem is the following
(see \cite{S} and the references therein):

\begin{theorem}
If an IFS $F$ on $(\x,d)$  
satisfies $\sup_{f\in F} Lip(f,d) <1$
and is compact, then it admits a Hutchinson attractor.
\end{theorem}

Roughly speaking, Lemma~\ref{abc3} says that, 
if compact IFSs $F,G$ are close to each other 
on a bounded subinvariant set, in the sense that 
each map $f$ from $F$ has a close neighbour $g\in G$, 
and vice-versa, then attractors of $F$ and $G$ are also close.

\begin{lemma}\label{abc3}
Let $\G:=\{g_i:i\in I\}$ and $\mH:=\{h_j:j\in J\}$ be 
two compact IFSs on a complete metric space $(\x,d)$ 
such that $Lip(\G,d) < 1$ and $Lip(\mH,d)<1$. 
Let $B \subseteq \x$ be a compact set such that 
$\G(B) \subseteq B$ and $\mH(B) \subseteq B$, 
and let $\delta>0$ satisfy
\begin{equation}\label{abc1}
    \forall_{i\in I}\;\exists_{j\in J}\;\forall_{x\in B} \;\;\;
    d(g_i(x),h_j(x))\leq\delta\quad \mbox{and}\quad 
    \forall_{j\in J}\;\exists_{i\in I}\;\forall_{x\in B}\;\;\;
    d(g_i(x),h_j(x))\leq\delta.
\end{equation}
Then
\begin{equation*}
    h(A_\G,A_\mH)\leq\frac{\delta}{1-\min\{Lip(\G,d),Lip(H,d)\}},
\end{equation*}
where $A_\G$ and $A_\mH$ are the Hutchinson attractors 
of $\G$ and $\mH$, respectively.
\end{lemma}

\begin{remark}
Given two compact IFSs $G$ and $H$ with attractors 
$A_G$ and $A_H$, there always exists a nonempty compact 
$B\subseteq\x$
such that $G(B)\subseteq B$ and $H(B) \subseteq B$. 
Indeed, since $G$ and $H$ are compact,  
the IFS $F\cup G$ is also compact, 
hence admits the attractor $A_{G\cup H}$. Furthermore, 
for any nonempty compact set $D\subseteq \x$, the set 
\[
	B:= \on{cl}\Big(D\cup \bigcup_{n\in\N}(G\cup H)^{(n)}(D)\Big) 
	=A_{G\cup H}\cup D\cup \bigcup_{n\in\N}(G\cup H)^{(n)}(D)
\]
is compact, and $G(B)\cup H(B) \subseteq B$.
\end{remark}

\begin{proof}(Of Lemma \ref{abc3}.)
By (\ref{abc1}), we can easily see that for any compact $D\subseteq B$,
\begin{equation}\label{abc2}
    h({\G(D),\mH(D)})\leq \delta.
\end{equation}
Without loss of generality suppose $\alpha=Lip(\G,d)\leq Lip(H,d)$. 
We will check inductively that for every $n\in\N$,
\begin{equation}\label{abc22}
   h(\G^{(n)}(B),H^{(n)}(B))\leq \delta\sum_{k=0}^{n-1}\alpha^k.
\end{equation}
The case $n=1$ of \eqref{abc22} is exactly (\ref{abc2}) for $D:=B$. 
Assume that the inequality \eqref{abc22} holds for some $n\in\N$. Then 
\[\begin{aligned}
    h(G^{(n+1)}(B),  H^{(n+1)}(B)) & 
    \leq h(\G(\G^{(n)}(B)),\G(H^{(n)}(B)))+ 
    h(\G(H^{(n)}(B)),\mH(H^{(n)}(B))) \\ 
    & \leq \alpha h(\G^{(n)}(B),\mH^{(n)}(B))+\delta\leq 
    \alpha \delta\sum_{k=0}^{n-1}\alpha^k+\delta=
    \delta\sum_{k=0}^{n}\alpha^k,
\end{aligned}\]
where the penultimate inequality follows from \eqref{abc2} 
for $D:=H^{(n)}(B)$, and the last inequality uses \eqref{abc22} 
for $n$. Thus \eqref{abc22} is true for $n+1$. 
Now from (\ref{abc22}) and the convergence of the Hutchinson iterates 
to the attractor, we get
\begin{equation*}
    h(A_\G,A_\mH)\leq \delta\sum_{k=0}^{\infty}\alpha^k =
    \frac{\delta}{1-\alpha}.
\end{equation*}
This completes the proof.
\end{proof}

\begin{lemma}\label{fildec1}
Let $\x$ be a metric space and $f_t$, $t\in[0,1]$, 
be a family of nonexpansive selfmaps of $\x$ such 
that the map
$[0,1]\ni t\mapsto f_t(x)$ is continuous  for every $x\in\x$.
Then the IFS $F:=\{f_t:t\in [0,1]\}$ is compact.
\end{lemma}

\begin{proof}
Take any nonempty and compact set $D \subseteq \x$. 
By Lemma \ref{fildec5}, the map 
$[0,1]\ni t\mapsto f_t(D)\in\K(\x)$ is continuous. 
This implies that
\[ 
F(D)=\overline{\bigcup_{t\in[0,1]}f_t(D)} =
\bigcup_{t\in[0,1]}f_t(D)
\]
is compact thanks to 
\cite[Corollary 2.20 chap. 2.1 p.42 and Theorem 2.68 chap. 2.2 p. 62]{HuPapa}. 
\end{proof}

Recall that any surjective isometry $g:\x\to\x$ 
of a real normed space is of the following form: 
\[
	g(x)=\hat{g}(x) + b=\hat{g}(x-x_*),
\]
where $\hat{g}:\x\to\x$ is a linear isometry, 
$b=g(0)\in\x$ and $x_*=g^{-1}(0)$
(cf. \cite{FlemingJamison} chap.1.3, Mazur--Ulam theorem).

\begin{lemma}\label{fildec2}
Let $\x$ be a real Banach space; 
let $g:\x\to \x$ be a surjective isometry; 
let $x_*=g^{-1}(0)$; and let $\hat{g}$ 
be the linear part of $g$.
For $t\in[0,1]$, set 
\[
	g_t(x):=tg(x)+x_*,\;\;x\in X.
\] 
The following statements hold:
\begin{itemize}
\item[(a)] For every $m\in\N$, $t_1,...,t_m\in[0,1]$ 
and for all $x\in \x$, we have
\begin{equation*}
	g_{t_1}\circ...\circ g_{t_m}(x)=t_1\cdots 
	t_m \,\hat{g}^{(m)}(x-x_*)+x_*.
\end{equation*}
\item[(b)] $g_1$ is periodic if and only if $\hat{g}$ is periodic, 
and their periods are the same.
\item[(c)] If $g_1$ is periodic, then the monoid 
generated by the IFS $G:=\{g_t:t\in[0,1]\}$ is compact.
\end{itemize}
\end{lemma}

\begin{proof}
By the preceding observations concerning surjective isometries, we have
\[
	g_{t_1}(x)= t_1\hat{g}(x-x_*)+x_*
\]
which gives us (a) for $m=1$. 
Suppose that (a) is true for some $m\in\N$. Then we have 
\[\begin{aligned}
	g_{t_1}\circ...\circ g_{t_m}\circ g_{t_{m+1}}(x) &
	=t_1\cdots t_m\hat{g}^{(m)}(t_{m+1}\hat{g}(x-x_*)+x_*-x_*)+x_* \\
	& =t_1\cdots t_{m+1}\hat{g}^{(m+1)}(x-x_*)+x_*
\end{aligned}\]
so we obtain (a) for $m+1$. This ends the proof of (a).

By (a), for every $m\in\N$ and $x\in \x$, we have
\[
	g^{(m)}_1(x)-x_*=\hat{g}^{(m)}(x-x_*).
\]
Hence if $g^{(m)}_1=\on{id}_\x$, then also 
$\hat{g}^{(m)}=\on{id}_\x$, and vice-versa. 
Thus (b) is true.

Now we prove (c).  By (a), each element of 
the desired monoid $\mathbb{M}(G)$, distinct 
from the identity map, is of the form
\[
g_{t_1}\circ...\circ g_{t_m}(x)= 
t_1\cdots t_m\hat{g}^{(m)}(x-x_*)+x_*=
t^i\hat{g}^{(i)}(x-x_*)+x_*=g_t^{(i)}(x)
\]
for some $i=1,...,p$ where 
$p$ is the period of $\hat{g}$ and 
$t:=\sqrt[i]{t_1\cdots t_m}$. Hence
\[
\mathbb{M}(G)=\{g_t^{(i)}:i=1,...,p,\;t\in[0,1]\}
\]
(note that $g_1^{p}=\on{id}_\x$). In particular, 
$\mathbb{M}(G)$ is the finite union of IFSs 
$\{g_t^{(i)}:t\in[0,1]\}$ over $i=1,...,p$, 
which are compact in view of Lemma \ref{fildec1}. 
Thus $\mathbb{M}(G)$ itself is compact.
\end{proof}

We now state the main result of this section, 
which shows that quite a wide class of IFS families 
possess a unique upper transition attractor. 
Note that statement (b) in Lemma \ref{fildec2} 
is intended to clarify the assumption on 
periodicity of the linear part of $g$ in  Theorem~\ref{thm:fildec15}.

\begin{theorem}\label{thm:fildec15} 
Let $\x$ be a real Banach space and let $g:\x\to\x$ be a surjective isometry on $\x$
with periodic linear part.  Consider the one-parameter family 
\[
F_t^g := F_t\cup\{g_t\}
\] 
on $\x$ with $t\in [0,1]$, where 
\[
F_t: =\{f_{(i,t)} \, : \, 1\leq i\leq N\}
\]
and
\[
g_t(x):=tg(x)+x_*,
\]
 where
 $x_* = g^{-1}(0)$ or, equivalently, $g_t(x_*) = x_*$ for all $t$. 
Assume that $F_t$ satisfies:
\begin{itemize}
\item[(i)] for any $i=1,...,N$ and $x\in\x$, the map 
$[0,1]\ni t\mapsto f_{(i,t)}(x)$ is continuous, and
\item[(ii)] $\sup\{\on{Lip}(F_t,||\cdot||):t\in[0,1]\}<1$. 
\end{itemize}
Then $t_0=1$ is a threshold for the one-parameter family $F^g_t$, and $F^g_t$ has a unique upper transition attractor. 
\end{theorem}

\begin{remark}\label{rem:com}  Two comments before the proof:

 First, Examples \ref{ex:two}, \ref{ex:three}, \ref{ex:four}, and \ref{ex:inf} below show that the assumptions in the hypothesis of Theorem~\ref{thm:fildec15} are essential.

Second, that the upper transition attractor in \cite[Proposition 8.1]{V} 
is unique is a direct consequence of Theorem~\ref{thm:fildec15}.

\end{remark}

\begin{proof}[Proof of Theorem~\ref{thm:fildec15}]

 Since
all functions in $F_t^g$ are contractions for $t<1$, the one-parameter family $F_t^g$ has an attractor $A_t$ for $t\in [0,1)$.  Since $g_t$ is a similarity with ratio greater than $1$ for $t>1$, the one-parameter family $F_t^g$ has no attractor for $t>1$.  Therefore $t_0=1$ is a threshold for $F_t^g$. In other words 
$ t_0 = \widehat{t_0} = 1$. 

From \cite[Proposition 8.1]{V} the existence of 
a unique upper transition attractor 
is equivalent to the uniform continuity of the map 
\begin{equation*}
[0,1)\ni t\mapsto A_{t}\in\K(X).
\end{equation*} 
Hence we will prove that this map is uniformly continuous.\\
$\;$\\
Step 1. Finding a nonempty and compact set $B$ so that 
\begin{equation}\label{filmarch1}
f_{(i,t)}(B)\subseteq B\;\;\mbox{and}\;\;g_t(B)\subseteq B
\end{equation}
for all $t\in[0,1]$ and $i=1,...,N$.\\ 

Consider the IFSs
\begin{equation}\label{fildec14} 
\begin{aligned}
F & :=\bigcup_{t\in[0,1]} F_t=\{f_{(i,t)}:i=1,...,N,\;t\in[0,1]\}=
\bigcup_{i=1,...,N}\{f_{(i,t)}:\;t\in[0,1]\} \\
G & :=\{g_t:t\in[0,1]\}.
\end{aligned}
\end{equation}
In view of Lemma \ref{fildec1}, the IFS $F$ is 
a finite union of compact IFSs. Hence $F$ is compact. 
Also, in view of Lemma \ref{fildec2} (c), 
the monoid $\mathbb{M}(G)$ is compact. Moreover, $G$ consists of 
nonexpansive maps and by assumption (ii) we have $Lip(F,||\cdot||)<1$. 
Then using \cite[Theorem 4.1]{S} (cf. also \cite[Remark 2.2]{S}), we see 
that the IFS $F\cup G$ has compact semiattractor $B$. In particular,
\[
	B=F(B)\cup G(B).
\]
Therefore (\ref{filmarch1}) holds.\\
$\;$\\
Step 2. An alternative description of the attractor $A_t$ of $F^g_t$.\\

Fix a real value $t\in[0,1)$. Clearly, 
\[
	Lip(F^g_t,||\cdot||)\leq\max\{t,Lip(F,||\cdot||)\}<1.
\] 
Hence $F^g_t$ generates a unique attractor $A_t$. Again using 
\cite[Theorem 4.1]{S} for IFSs $F_t$ and $\{g_t\}$, we see that 
$A_t$ can be viewed as the attractor of the IFS
\begin{equation}\label{fildec12}
M_t:=\{g_t^{(m)}\circ f_{(i,t)}:i=1,...,N,\;m=0,1,2,...\}
\end{equation}
where $g_t^{(0)}=\on{id}_{\x}$. Note that the assumptions of 
\cite[Theorem 4.1]{S} will be satisfied if we observe that the monoid 
\[
	\mathbb{M}(\{g_t\})=\{g_t^{(m)}:m=0,1,2,...\}
\] 
is compact. This is the case as it is a subset of a compact IFS 
$\mathbb{M}(G)$ considered in Step 1. (Alternatively, we can observe 
that $\mathbb{M}(\{g_t\})$ is compact by using the fact 
$\on{Lip}(g_t)\leq t<1$.)\\
Moreover, in view of (\ref{filmarch1}), we see that $A_t \subseteq B$.\\
$\;$\\
Step 3. Uniform continuity of the map $[0,t_0]\ni t\mapsto A_{t}$, 
where $t_0\in[0,1)$.\\

Fix any $t_0\in[0,1)$. Clearly,
\[
\sup\{\on{Lip}(F^g_t,||\cdot||):t\in[0,t_0]\}\leq 
\max\{t_0,Lip(F,||\cdot||)\}<1.
\]
Hence the assumptions of \cite[Theorem 2.6]{J} are satisfied. 
This means that the map $[0,t_0]\ni t\mapsto A_{t}$ is continuous. 
As $[0,t_0]$ is compact, it is uniformly continuous.\\
$\;$\\
Step 4. Uniform continuity of the map $[0,1)\ni t\mapsto A_{t}.$\\

The idea in the proof below is that if both $t,s<1$ are appropriately 
less than $1$, then we make use of uniform continuity proved 
in Step 3, whereas if $s,t$ are both 
sufficiently close to $1$, then for each map of the form
$g_t^m\circ f_{(i,t)}$ we will find sufficiently close neighbour 
$g_s^k\circ f_{(i,s)}$ (where $k$ will be appropriately chosen), and vice-versa. 
Then we will make use of Lemma \ref{abc3}.\\

Let $\hat{g}$ be the linear part of $g$. Then 
by Lemma \ref{fildec2}, we see that for every $m\in\N$
and $x\in X$, we have:
\begin{equation}\label{fildec22}
g_t^{(m)}(x) = 
t^m\hat{g}^{(m)}(x-x_*)+x_*.
\end{equation}
Let $p$ be the period of $\hat{g}$. 
Take any $\varepsilon>0$ and choose $r\in(0,1)$ such that
\begin{equation}\label{fildec4}
(1-r^p)\cdot(\on{diam}(B\cup\{0\})+||x_*||)<\frac{\varepsilon}{2}.
\end{equation}
Then choose $\delta>0$ such that:
\begin{itemize}
\item[(a)] for $s,t\in[0,r]$, 
if $|t-s|<\delta$, then $h(A_{t},A_{s})<\varepsilon$;
\item[(b)] for $s,t\in[0,1]$, 
if $|t-s|<\delta$, then 
\[
{\sup\{||f_{(i,t)}(x)-f_{(i,s)}(x)||:i=1,...,N,\;x\in B\} 
<\frac{\varepsilon}{2}};
\]
\item[(c)] $(1-(r-\delta)^p)) \cdot 
(\on{diam}(B\cup\{0\})+||x_*||) \leq
\frac{\varepsilon}{2}.$
\end{itemize}
The choice of $\delta$ is possible by Step 3 (for item (a)), 
by Lemma \ref{fildec5} (for item (b)) and condition 
(\ref{fildec4}) (for item (c)).\\
Now choose $s,t\in[0,1)$ such that $|s-t|<\delta$. 
If $s,t\leq r$, then $h(A_t,A_s)\leq\e$ in view of (a). 
Hence assume that 
\begin{equation}\label{filaa3}
\max\{s,t\}\geq r.
\end{equation} 
Take any $i=1,...,N$ and $m=0,1,2,...$, and 
let $m',l'$ be such that $m=pm'+l'$, and $l'=0,...,p-1$. 
Then let $k'$ be the least nonnegative integer such that
\begin{equation*}
    s^{pk'+l'}\leq t^{pm'+l'}
\end{equation*}
and set $k:=pk'+l'$. We will show that
\begin{equation}\label{abc7}
    |t^m-s^k|\leq 1-(r-\delta)^p.
\end{equation}
Using $s^k\leq t^m<s^{k-p}$, we have
\[\begin{aligned}
|t^m-s^k| &=t^m-s^k\leq \min\{1,s^{k-p}\}-s^k= 
\min\{1-s^k,s^{k-p}(1-s^p)\} \\
&\leq\left\{ \begin{array}{ccc}1-s^{l'}&\mbox{if}&k'=0\\
s^{k-p}(1-s^p)&\mbox{if}&k'\geq 1\end{array}\right.\\ 
&\leq 1-s^p\leq 1-(r-\delta)^p, 
\end{aligned}\]
where the last iequality follows from  $r-\delta\leq s$ 
(thanks to (\ref{filaa3})). Thus we have shown (\ref{abc7}).

Now fix $i=1,...,L$ and choose any $x\in B$. 
Assume that $m\geq 1$ (which also implies $k\geq 1$). 
Set $z_t:=f_{(i,t)}(x)-x_*$ and $z_s:=f_{(i,s)}(x)-x_*$. 
Then by (b) and (c) from the choice of $\delta$, we have
\[
||z_t-z_s||=||f_{(i,t)}(x)-f_{(i,s)}(x)||<\frac{\varepsilon}{2}
\]
and
\[
||z_s||\leq||f_{(i,s)}(x)-0||+||x_*||\leq\on{diam}(B\cup\{0\})+
||x_*||\leq\frac{\varepsilon}{2}\cdot (1-(r-\delta)^p)^{-1}.
\] 
Hence by (\ref{fildec22}) and (\ref{abc7}), 
and the fact that $\hat{g}^{(p)} =\on{id}_\x$, we have
\begin{equation*} \begin{aligned}
  ||g^{(m)}_t\circ f_{(i,t)}(x) 
  -g^{(k)}_s\circ f_{(i,s)}(x)|| &=
  ||t^m\hat{g}^{(m)}(f_{(i,t)}(x)-x_*)+ x_*-s^k
  \hat{g}^{(k)}(f_{(i,s)}(x)-x_*)-x_*||
\\ 
  & =||t^m\hat{g}^{(m)}(z_t)-s^k\hat{g}^{(k)}(z_s)||\\
  &=
  ||t^m\hat{g}^{(pm'+l')}(z_t)-s^k\hat{g}^{(pk'+l')}(z_s)|| 
\\
  & =||t^m\hat{g}^{(l')}(z_t)-s^k\hat{g}^{(l')}(z_s)|| \\
  & \leq
  ||t^m\hat{g}^{(l')}(z_t)-t^m\hat{g}^{(l')}(z_s)|| + 
  ||t^m\hat{g}^{(l')}(z_s)-s^k\hat{g}^{(l')}(z_s)|| 
\\
  & =t^m\cdot ||\hat{g}^{(l')}(z_t-z_s)||+
  |t^m-s^k|\cdot ||\hat{g}^{(l')}(z_s)||\\
  &=
  t^m\cdot ||z_t-z_s||+|t^m-s^k|\cdot||z_s|| 
\\
  & \leq 
  ||z_t-z_s||+(1-(r-\delta)^p)\cdot ||z_s||\\
  &<\varepsilon.
\end{aligned}\end{equation*}

When $m=0$ (and consequently $k=0$), we also have
\begin{equation*}
  ||g_t^{(m)}\circ f_{(i,t)}(x) - 
  g_s^{(k)} \circ f_{(i,s)}(x)|| =
  ||f_{(i,t)}(x)-f_{(i,s)}(x)|| < 
  \frac{\varepsilon}{2}.
\end{equation*}
Similar reasoning works when the roles of $s$ and $t$ 
are switched. Hence we see that condition (\ref{abc1}) 
from Lemma \ref{abc3} is satisfied for IFSs $M_t$ and $M_s$, 
whose attractors are $A_{t}$ and $A_{s}$, 
respectively (for definitions of $M_t$ and $M_s$, 
see (\ref{fildec12})).
Thus, using Lemma \ref{abc3}, and the fact that 
\[
Lip(M_s,||\cdot||), Lip(M_t,||\cdot||) \leq 
Lip(F,||\cdot||)<1
\]
(recall definition of $F$ in (\ref{fildec14}) 
and notice that $Lip(g^{(m)}\circ f) = Lip(f)$ 
for $f\in F$), we get
\begin{equation*}
  h(A_{t},A_{s})\leq \frac{\e}{1-Lip(F,||\cdot||)}.
\end{equation*}
W conclude that the map $[0,1)\ni t\mapsto A_t$ 
is uniformly continuous.
\end{proof}

 As mentioned in Remark \ref{rem:com}, Examples \ref{ex:two}, \ref{ex:three}, \ref{ex:four}, and \ref{ex:inf} show that each
of the assumptions in the hypothesis of Theorem~\ref{thm:fildec15} is necessary.
If any assumption is removed, not only does the family $F_t^g$ not have a unique upper transition attractor, but it  may have no upper transition attractor at all.
In particular, Example~\ref{ex:inf} provides an
infinite dimensional one-parameter family where the function $g$ is not periodic
and the one-parameter family has no upper transition attractor.   For a more 
restricted one-parameter family, however, this assumption may 
not be necessary; see Question~\ref{ques:per}.
	
 \begin{example}\label{ex:two} [The assumption that $[0,1] \ni t \mapsto  f_{(i,t)}(x)$
is continuous is necessary.] 

 Let $F_t^g := \{f_{t}, g_{t} \}$ be a one-parameter family on $\R$, 
where $f_{t}(x) = 
tx/2 +(2-t)/(1-t), t\in [0,1)$, $f_1$ is any continuous function, and $ g_{t}(x) = tx, t\in [0,1]$.  Here $F_t^g$ satisfies the assumptions of  
Theorem~\ref{thm:fildec15} except that $[0,1] \ni t \mapsto  f_{t}(x)$  is not continuous at $t=1$ for any $x$.  The fixed point $q_t$ of $f_t$ is $q_t = 2/(1-t) \rightarrow \infty$ as $t \rightarrow 1$.  Since $q_{t} \in A_t$,
 the limit $\lim_{t\rightarrow 1} A_t$ does not exist.   
\end{example}

\begin{example}\label{ex:three} [The assumption that $x^* =  g^{-1}(0)$ is
necessary.]     

Let $F_t^g := \{f_{t},g_{t} \}$, where $g(x) = x, \, g_t(x)= tx+1$, and
$f_{t}(x) = tx/2 +1$. Here $F_t^g$ satisfies the assumptions of  
Theorem~\ref{thm:fildec15} except that $x^* = 1 \neq g^{-1}(0)$.  
In this case the fixed point $q_{t}$ of $g_t$ is $q_{t} = 1/(1-t) \rightarrow \infty$ as $t \rightarrow 1$.  Since $q_{t} \in A_t$, the limit $\lim_{t\rightarrow 1} A_t$ does
not exist.   
\end{example}

\begin{example}  \label{ex:four} [The assumption that $\sup\{\on{Lip}(F_t,||\cdot||):t\in[0,1]\}<1$ is
necessary.]

 On $\mathbb R$, let $F_t^g := \{f_{t}, g_t \}$, where $g_t(x) = -tx, f_t(x) = -tx+t +1$.  
(This is Example~\ref{ex:five} from Section~\ref{sec:uls}.) Here $F_t^g$ satisfies the assumptions of  Theorem~\ref{thm:fildec15} except that $\lim_{t\rightarrow 1} Lip(f_t,||\cdot||) =1$. For $t\in (0,1)$ we have  $A_t = [-t/(1-t), 1/(1-t)]$; therefore
 $\lim_{t\rightarrow 1} A_t$ does not exist.   
\end{example}

\begin{example} \label{ex:inf}
[The assumption that $g$ is periodic is necessary.] 

Let $\x:=\ell^\infty(\mathbb{C})$ denotes
the real Banach space of all bounded complex sequences, 
endowed with the supremum norm. For $k\in\N$, 
set $\alpha_k:=\frac{\pi}{2k}$, 
and define $g:\x\to \x$ by
\begin{equation*}
g((x_k)):=\left(x_ke^{i\alpha_k}\right),
\end{equation*}
that is, each coordinate $x_k$ is rotated around 
the origin by angle $\alpha_k$. 
Next define $f:\x\to \x$ by
\begin{equation*}
f((x_k)):=\left(\frac{1}{4}(x_k-1)\right).
\end{equation*}
Observe that $f(\textbf{1})=\textbf{0}$, 
where $\textbf{1}$ and $\textbf{0}$ are sequences 
of ones and zeroes, respectively. 
For $t\in(0,1]$, define
\begin{equation*}
  g_t((x_k)):=tg((x_k)) = 
  \left(tx_ke^{i\alpha_k}\right)
\end{equation*}
and 
\begin{equation*}
  f_t((x_k)) :=tf((x_k))+ \textbf{1} =
  \left(\frac{t}{4}x_k+1-\frac{t}{4}\right).
\end{equation*}
Clearly, the map 
$[0,1]\ni t\mapsto f_t(x)$ 
is continuous for every $x\in\x$, 
$Lip(f_t)=\frac{t}{4}$ and 
$g^{-1}(\textbf{0})=\textbf{0}$. 
Hence, setting $F_t:= \{f_t\}$, 
all assumptions of Theorem \ref{thm:fildec15} are satisfied except that $\hat{g}$ 
(which here coincides with $g$) is not periodic.  
We will now show that $F^g_t$ does not have
any upper transition attractor.
\\

Let
\begin{equation}\label{fildec17}
D:=\{\textbf{0}\}\cup \bigcup_{m=0}^\infty 
\overline{B}\Big(\Big(\frac{3}{4}t^m e^{im\alpha_k}\Big),
\frac{1}{4}t^m\Big)
\end{equation}
where $\overline{B}(\cdot,\cdot)$ denotes 
the closed ball in $\x$, where the first coordinate is the center and
 the second coordinate is the radius. 
We first show that for every $t\in[0,1]$, 
the attractor $A_t$ of $F^g_t:= \{f_t,g_t\}$ 
is a subset of $D$. Clearly, the set 
$D \subseteq \overline{B}(\textbf{0},1)$, 
and it is easy to see that
\begin{equation*}
f_t(\overline{B}(\textbf{0},1)) \subseteq 
\overline{B}\Big(\Big(\frac{3}{4}\Big),
\frac{1}{4}\Big) \subseteq D,
\end{equation*}
(where $\Big(\frac{3}{4}\Big)$ is the constant sequence
whose coordinates equal $\frac{3}{4}$). Hence 
\[
f_t(D) \subseteq D.
\]
On the other hand, for every $m=0,1,2,....$, we have
\[
g_t\Big(\overline{B}\Big(\Big(\frac{3}{4}t^m e^{im\alpha_k}\Big),
\frac{1}{4}t^m\Big)\Big) =
\overline{B}\Big(\Big(\frac{3}{4}t^{m+1}e^{i(m+1)\alpha_k}\Big),
\frac{1}{4}t^{m+1}\Big) \subseteq D
\]
and $g_t(\textbf{0})=\textbf{0}$;
so we also have
\[
g_t(D)\subseteq D.
\]
Altogether we have $F_t^g(D)\subseteq D$. 
As $D$ is closed, we get (\ref{fildec17}).\\
Now since the sequence $\textbf{1}$ is 
the fixed point of $f_t$, it belongs 
to the attractor $A_t$, and hence also 
\begin{equation}\label{fildec21}
\Big(t^{m}e^{im\alpha_k}\Big) =
g_t^{(m)}(\textbf{1}) \in A_t
\end{equation}
for every $m\in\N$.

We are ready to prove that $(F^g_t)$ does not 
generate any upper transition attractor, that is, 
there is no sequence $t_n\in[0,1)$ with $t_n\to 1$ 
so that $(A_{{t_n}})$ converges. 
First observe that it is enough to prove that 
\begin{equation}\label{filaa}
\forall_{s\in[\frac{1}{2},1)}\; 
\exists_{t_0<1}\;\forall_{t\in[t_0,1)}\; 
h(A_t,A_s)\geq\frac{1}{2}.
\end{equation}
Indeed, suppose that (\ref{filaa}) holds, and 
for some sequence $(t_n)\subseteq [0,1)$ 
converging to $1$ we have that $(A_{t_n})$ is convergent. 
Then $(A_{t_n})$ is a Cauchy sequence in $\K(X)$ 
and we can find $n_0\in\N$ so that 
$h(A_{t_{n_0}},A_{t_n})<\frac{1}{2}$ for all $n\geq n_0$ 
and $t_{n_0}\geq \frac{1}{2}$. On the other hand, 
setting $s:=t_{n_0}$ and using (\ref{filaa}), 
we can find $n\geq n_0$ with 
$h(A_{t_n},A_{t_{n_0}})\geq \frac{1}{2}$, 
which gives a contradiction. 

We will now prove (\ref{filaa}).
Choose any $s\in[\frac{1}{2},1)$, 
and find the least $k_0\in\N$ such that 
$s^{k_0}<\frac{1}{2}$. As $s\geq \frac{1}{2}$, 
we see that $s^{k_0}\geq \frac{1}{4}$. 
Since $1-s^{k_0}>\frac{1}{2}$, we can find $t_0<1$ 
such that for $t\in[t_0,1)$ we have
\begin{equation}\label{filaaa}
t^{2k_0}-s^{k_0}>\frac{1}{2}.
\end{equation}
Choose any $(x_k)\in A_s$. By the definition of $D$ 
(see (\ref{fildec17})) and the fact that $A_t\subset D$, 
we can consider three cases: \\ \\
Case 1. $(x_k)\in \overline{B}\Big(\Big(\frac{3}{4}s^m e^{im\alpha_k}\Big), 
\Big(\frac{1}{4}s^m\Big)\Big)$ for some $m\leq k_0$. 

Since  $m\alpha_{k_0}\leq k_0\frac{\pi}{2k_0}=\frac{\pi}{2}$, we have
\[
t^{2k_0}\leq \Big|t^{2k_0}+\frac{3}{4}s^m e^{im\alpha_{k_0}}\Big|\leq 
\Big|t^{2k_0}+x_{k_0}\Big|+
\Big|-x_{k_0}+\frac{3}{4}s^m e^{im\alpha_{k_0}}\Big|\leq 
\Big|t^{2k_0}+x_{k_0}\Big|+\frac{1}{4},
\]
so by (\ref{filaaa}) we get
\[
\Big|t^{2k_0}+x_{k_0}\Big|\geq t^{2k_0}-\frac{1}{4}\geq 
t^{2k_0}-s^{k_0}>\frac{1}{2}.
\]
Case 2. $(x_k)\in \overline{B}\Big(\Big(\frac{3}{4}s^m e^{im\alpha_k}\Big),
\Big(\frac{1}{4}s^m\Big)\Big)$ for some $m\geq k_0$. 

Since $t^{2k_0}>s^{k_0}\geq s^m$, we have
\[\begin{aligned}
t^{2k_0}-\frac{3}{4}s^{k_0} &
\leq t^{2k_0}-\frac{3}{4}s^{m}\leq \Big|t^{2k_0}+
\frac{3}{4}s^{m} e^{im\alpha_{k_0}}\Big | \leq  
\Big|t^{2k_0}+x_{k_0}\Big|+
\Big|-x_{k_0}+\frac{3}{4}s^{m}e^{im\alpha_{k_0}}\Big| 
\\ & 
\leq \Big|t^{2k_0}+x_{k_0}\Big|+ 
\frac{1}{4}s^m\leq \Big|t^{2k_0}+x_{k_0}\Big|+
\frac{1}{4}s^{k_0}.
\end{aligned}\]
Thus by (\ref{filaaa}),
\[
\Big|t^{2k_0}+x_{k_0}\Big|\geq 
t^{2k_0}-\frac{3}{4}s^{k_0}- 
\frac{1}{4}s^{k_0}>\frac{1}{2}.
\]
Case 3. $(x_k)=\textbf{0}$. 

In this case
$$
\Big|t^{2k_0}+x_{k_0}\Big|=t^{2k_0}>\frac{1}{2}.
$$
Summing up, we have that 
\begin{equation*}
\Big|\Big|\Big(t^{2k_0}e^{i\frac{2k_0}{2k}\pi}\Big)-(x_k)\Big|\Big|\geq 
\Big|t^{2k_0}e^{i\frac{2k_0}{2k_0}\pi}-x_{k_0}\Big|=
|-t^{2k_0}-x_{k_0}|=|t^{2k_0}+x_{k_0}|>\frac{1}{2}.
\end{equation*}
By (\ref{fildec21}) we see that 
$\Big(t^{2k_0} e^{i\frac{2k_0}{2k}\pi}\Big)\in A_t$, 
so the above shows that
\[
h(A_t,A_s)\geq \inf_{(x_k)\in A_s} 
\Big|\Big|\Big(t^{2k_0}e^{i\frac{2k_0}{2k}\pi}\Big)-
(x_k)\Big|\Big|\geq \frac{1}{2}
\]
and the proof of (\ref{filaa}) is complete.
\end{example}

\section{Open Problems} \label{sec:op}
 
Examples~\ref{ex:S1}, \ref{ex:S2}, and \ref{ex:S3} show that an IFS with an attractor need not be contractive.  In Example~\ref{ex:S2} no function in the IFS $F$ is a contraction.  In fact,  with respect to any equivalent metric $d$ on the circle, $Lip(f,d)>1$ for all $f\in F$.  This is not the case in Example~\ref{ex:S3}.  It can be asked whether such a strong counterexample exists for $\R^n$.  

\begin{question}   Is there an example of an IFS $F$ on $\R^n$ that has an attractor $A$ with basin $\R^n$ but (1) $A$ is not the attractor of any proper subset of $F$ and (2) with respect to any metric $d$ equivalent to the Euclidean metric we have $Lip(f,d)>1$ for all $f\in F$.
\end{question}

For a large class of one-parameter IFS families, 
Theorem~\ref{thm:fildec15} guarantees the existence of 
a unique upper transition attractor $A^{\bullet}$ 
such that $A^{\bullet} = \lim_{t \rightarrow t_0} A_t$ 
at a threshold $t_0$. The theorem, however, assumes 
that the linear part of the special function $g$ is periodic.  
Example~\ref{ex:inf} shows that, in general, 
the assumption of periodicity of the linear part cannot be dropped. 
But the underlying space in that example is 
a non-separable infinite dimensional space.  

\begin{question}  \label{ques:per}
Can the assumption of periodicity of the linear part 
of the function $g$ in Theorem~\ref{thm:fildec15} 
be dropped assuming a less exotic space? In particular, can the assumption be dropped for a one-parameter similarity
family with threshold $t_0$ satisfying the following properties: 
\begin{itemize} 
\item All $f_t\in F_t$ are contractions for
$t\in [0,t_0]$,  $g_t$ is a contraction for
$t\in [0,t_0)$  and $Lip(g_{t_0}) = 1$, and
\item the unique fixed point of each $f_t\in F_t$ and $g_t$ is independent of $t \in [0,t_0)$.  
\end{itemize}

\end{question}

In \cite[Theorem 8.2]{V} relationships between 
the upper and lower transition attractors are given for a special type 
of one-parameter family. It can be asked whether the same 
relationships hold in a more general setting. In particular:

\begin{conjecture}
If $F_t$ satisfies properties (H1), (H2), (H3) of Section~\ref{sec:uls} 
and if $A_{\bullet} = A^{\bullet}$ for some 
upper transition attractor of $F_t$, then $A^{\bullet}$ is the unique 
upper transition attractor of $F_t$ and  $A^{\bullet}$ is an attractor of $F_1$.  
\end{conjecture}

Recall that in a metric space $(\x,d)$, 
a {\it segment} with ends $x,y\in \x$ is defined by 
$[x,y] := \{z\in\x: d(x,z)+d(z,y)=d(x,y)\}$.
A set $S\subseteq \x$ is {\it metrically convex} if
 $[x,y]\subseteq S$ for all $x,y\in S$. 
The {\it metrically convex hull} of $S\subseteq \x$ is 
$\operatorname{conv}_d S := \bigcup_{x,y\in S} [x,y]$.

\begin{conjecture}
If the functions in $F_t$ map metrically convex sets 
onto metrically convex sets, then
the metrically convex hulls of $A_{\bullet}$ 
and $A^{\bullet}$ in $(\x,d)$ coincide: 
$\operatorname{conv}_d A_{\bullet} = 
\operatorname{conv}_d A^{\bullet}$.
\end{conjecture}

\section*{Acknowledgement}
The contribution of the second author to the research in this paper was done while she was at University College Dublin.

\end{document}